\documentclass[a4paper]{article}
\usepackage[american]{babel}
\usepackage{a4wide}
\usepackage{graphicx}
\usepackage{amssymb,amsfonts,amsmath,amsthm,mathrsfs}

\usepackage{tikz}
\newcommand*\circled[1]{
	\tikz[baseline=(char.base)]{
	\node[anchor=text,shape=circle,draw,inner sep=-1pt](char){$#1$\strut};}}

\newcommand{\A}{\mathbf{A}}
\newcommand{\B}{\mathbf{B}}
\newcommand{\abs}[1]{\left\vert#1\right\vert}
\newcommand{\ds}{\displaystyle}
\newcommand{\eq}{\mathsf{e}}
\newcommand{\f}{\mathbf{f}}
\newcommand{\p}{\mathsf{p}}
\newcommand{\q}{\mathsf{q}}

\newcommand{\rhomax}{\rho_\text{max}}
\newcommand{\Tmax}{T_\text{max}}
\newcommand{\unit}[1]{\textup{#1}}
\newcommand{\vmax}{V_\text{max}}

\theoremstyle{plain}\newtheorem{theorem}{Theorem}[section]
\theoremstyle{remark}\newtheorem{remark}[theorem]{Remark}
\newenvironment{acknow}[1][Acknowledgement]{\textbf{#1.} }{
}

\graphicspath{{./}{figures/}}

\title{Fundamental diagrams in traffic flow: the case of heterogeneous kinetic models}
\author{Gabriella Puppo \\
		{\small\it Dipartimento di Scienza e Alta Tecnologia} \\[-1mm]
		{\small\it Universit\`a  dell'Insubria} \\[-1mm]
		{\small\it Via Valleggio 11, 22100 Como, Italy} \\[5mm]
		Matteo Semplice \\
		{\small\it Dipartimento di Matematica ``G. Peano''} \\[-1mm]
		{\small\it Universit\`a di Torino} \\[-1mm]
		{\small\it Via Carlo Alberto 10, 10123 Torino, Italy} \\[5mm]	
		Andrea Tosin \\
		{\small\it Istituto per le Applicazioni del Calcolo ``M. Picone''} \\[-1mm]
		{\small\it Consiglio Nazionale delle Ricerche} \\[-1mm]
		{\small\it Via dei Taurini 19, 00185 Rome, Italy} \\[5mm]
		Giuseppe Visconti \\
		{\small\it Dipartimento di Scienza e Alta Tecnologia} \\[-1mm]
		{\small\it Universit\`a dell'Insubria} \\[-1mm]
		{\small\it Via Valleggio 11, 22100 Como, Italy}
	   }
\date{}
	
\begin{document}

\maketitle

\begin{abstract}
Experimental studies on vehicular traffic provide data on quantities like density, flux, and mean speed of the vehicles. However, the diagrams relating these variables (the \emph{fundamental} and \emph{speed} diagrams) show some peculiarities not yet fully reproduced nor explained by mathematical models. In this paper, resting on the methods of kinetic theory, we introduce a new traffic model which takes into account the heterogeneous nature of the flow of vehicles along a road. In more detail, the model considers traffic as a mixture of two or more populations of vehicles (e.g., cars and trucks) with different microscopic characteristics, in particular different lengths and/or maximum speeds. With this approach we gain some insights into the scattering of the data in the regime of congested traffic clearly shown by actual measurements.

\medskip


\medskip

\end{abstract}

\section{Introduction}
Prediction and control of traffic have become an important aspect in the modern world. In fact, the necessity to forecast the depletion time of a queue or to optimize traffic flows, thereby reducing the number of accidents, has arisen following the increase of circulating vehicles.

In the current mathematical literature, three different approaches are mainly used to model traffic flow phenomena. \emph{Microscopic} models look at vehicles as single entities of traffic and predict, using a system of ordinary differential equations, the evolution of their position and speed (namely, the microscopic states characterizing their dynamics) regarded as time dependent variables. In these models, the acceleration is prescribed for each vehicle as a function of time, position, and speed of the various entities of the system, taking also into account mutual interactions among vehicles. For example, in the well known \emph{follow-the-leader} theory each vehicle is assumed to adapt its speed to the one of the leading vehicle based on their instantaneous relative speed and mutual distance, see~\cite{BrackstoneMcDonald,gazis1961OR,Pipes,ZhangMultiphase}. On the opposite end, \emph{macroscopic} models provide a large-scale aggregate point of view in which the focus is not on each single particle of the system. In this case, the motion of the vehicles along a road is described by means of partial differential equations inspired by conservation and balance laws from fluid dynamics, following the seminal works~\cite{lighthill1955PRSL,richards1956OR}. Improvements and further evolutions of such a basic macroscopic description of traffic have been proposed over the years by several authors, from the classical mechanically consistent restatement of second order models~\cite{aw2000SIAP} to applications to road networks thoroughly developed in the book~\cite{garavello2006BOOK}. More refined macroscopic models provide flux-density relations which depend on different states of the flow, e.g. see~\cite{Lebacque03}. In the middle, \emph{mesoscopic} (or \emph{kinetic}) models are based on a statistical mechanics approach, which still provides an aggregate representation of the traffic flow while linking macroscopic dynamics to pairwise interactions among vehicles at a smaller microscopic scale. These models will be the main reference background of the present paper. 

Kinetic (mesoscopic) models, first introduced in~\cite{paveri1975TR,prigogine1961PROC,prigogine1971BOOK}, are based on the Boltzmann equation that describes the statistical behavior of a system of particles. From the kinetic point of view, the system is again seen as the resultant of the evolution of microscopic particles, with given microscopic position and speed, but its representation is provided in aggregate terms by a statistical distribution function, whose evolution is described by integro-differential equations. Compared to microscopic models, the kinetic approach requires a smaller number of equations and parameters. On the other hand, unlike macroscopic models, at the mesoscopic scale the evolution equations do not require an a priori closure law: the flow is provided by the statistical moments of the kinetic distribution function over the microscopic states. Moreover, kinetic models are a quite natural way to bridge microscopic causes and macroscopic effects. They have also been extended to include multilane traffic flow~\cite{klar1999SIAP-1,klar1999SIAP-2}, flows on networks~\cite{fermotosin2015} and control problems~\cite{herty2007M2AS}, to name but just a few applications. Also, kinetic models have been proposed to derive macroscopic equations, see~\cite{HertyIllner08,HertyIllner12}.

For an overview of vehicular traffic models at all scales, the interested reader is referred e.g., to the review papers~\cite{klar2004BOOKCH,piccoli2009ENCYCLOPEDIA} and references therein.

In this paper we propose a multipopulation kinetic model for traffic flow, which draws inspiration from the ideas presented in~\cite{benzoni2003EJAM} for macroscopic models, recast in the frame of discrete-velocity kinetic models~\cite{delitala2007M3AS,fermo2013SIAP}. The main goal of this paper is to study fundamental diagrams, computed from moments of equilibrium solutions of the kinetic equations.  In particular, considering traffic flow as a mixture of populations with different microscopic characteristics helps to explain the experimentally observed scattering of fundamental diagrams in the phase of congested traffic. With this approach, scattered data in the congested phase are naturally {\em predicted} by the model, by taking into account the  macroscopic variability of the flux and mean speed at equilibrium due to the heterogeneous composition of the ``mixture''. This conclusion is reached without invoking further elements of microscopic randomness of the system: for example in~\cite{fermo2014DCDSS}, which inspired the present work, the explanation for the phase transition appeals to the stochasticity of the drivers' behavior and to the consequent variability of the {\em microscopic} speeds at equilibrium. Moreover, the models proposed here and in~\cite{delitala2007M3AS,fermo2014DCDSS} predict a sharp phase transition between the free and the congested phases of traffic, with a sharp capacity drop across the phase transition. We wish to stress that we do not propose a model that interpolates experimental data. Rather, we use experimental data to validate the model we propose.

In literature, a variety of multiphase models have been introduced in order to reflect the features of traffic, for a review see~\cite{klarReview} and references therein. The heterogeneity of traffic flow composition is often described by considering two or more classes of drivers with different behavioural attributes, see~\cite{LebacqueGSOM,MendezVelasco13}; here the heterogeneity will be described also by introducing two or more classes of vehicles with different physical features, as in~\cite{benzoni2003EJAM} for macroscopic models.

In more detail, the structure of the paper is as follows: in Section~\ref{sec:fund_diag} we briefly review the role of fundamental diagrams in vehicular traffic practice. Next, in Section~\ref{sec:1pop} we describe the discrete-velocity kinetic model developed in~\cite{delitala2007M3AS,fermo2013SIAP} by focusing on its spatially homogeneous version, which represents the mathematical counterpart of the experimental setting in which traffic equilibria and fundamental diagrams are measured. In Section~\ref{sec:2pop} we first review the multi-population macroscopic model~\cite{benzoni2003EJAM} and then introduce our new two-population kinetic model, proving in particular its consistency with the original single-population model and describing how to compute equilibrium solutions. Then in Section~\ref{sec:fundamental} we present and analyze the resulting fundamental diagrams, and we end in Section~\ref{sec:conclusions} with comments and perspectives.

\section{Fundamental diagrams}
\label{sec:fund_diag}
In this section we present a brief description of some basic tools for the analysis of traffic problems, namely the diagrams which relate the macroscopic flux and mean speed to the vehicle density in homogeneous steady conditions. The qualitative structure of such diagrams is defined by the properties of different regimes, or \emph{phases}, of traffic as outlined in the following.

\begin{figure}[!t]
\centering
\includegraphics[width=0.45\textwidth,height=0.2\textheight]{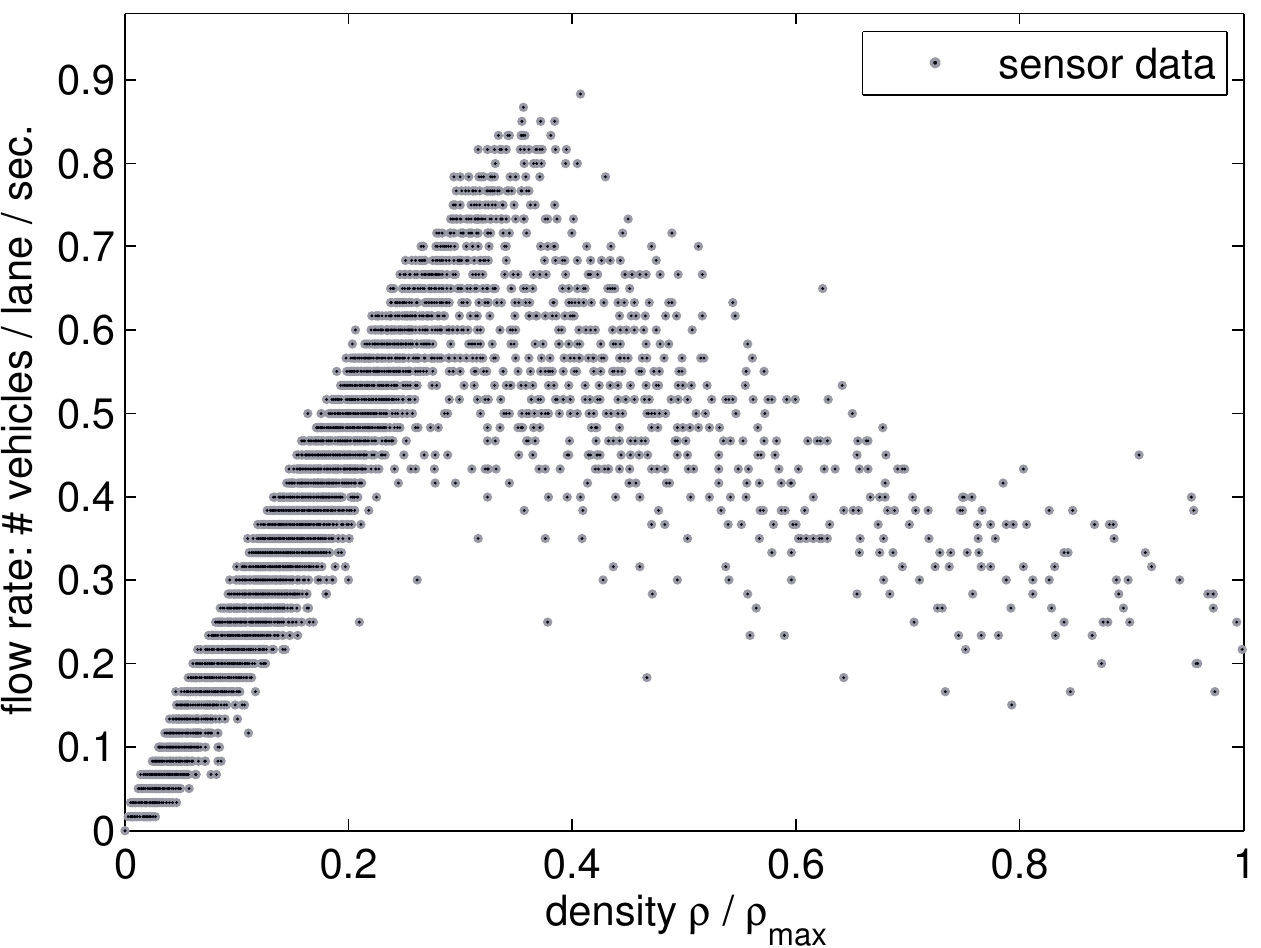}
\includegraphics[width=0.45\textwidth,height=0.2\textheight]{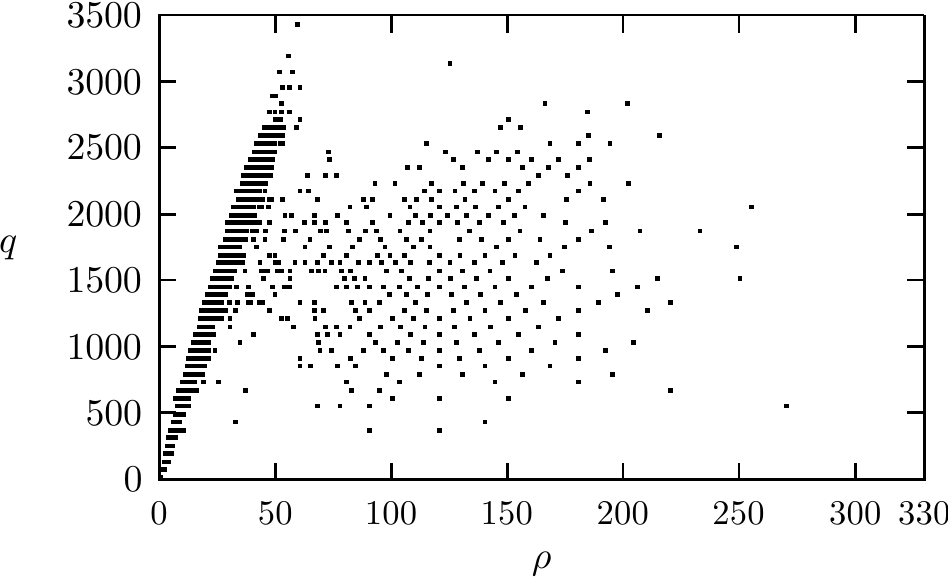}
\caption{Fundamental diagrams obtained from experimental data. Left: measurements provided by the Minnesota Department of Transportation in 2003, reproduced by kind permission from Seibold et al.~\cite{seibold2013NHM}. Right: experimental data collected in one week in Viale del Muro Torto, Rome, Italy, from~\cite{piccoli2009ENCYCLOPEDIA}.}
\label{fig:exp_diag}
\end{figure}

\begin{description}
\item[Flux-density diagrams] Also called \emph{fundamental diagrams}, they report the flow rate of vehicles as a function of the traffic density $\rho$, which can be defined as the number of vehicles  per kilometer (Fig.~\ref{fig:exp_diag}, right), or as a function of the normalized density (Fig.~\ref{fig:exp_diag}, left). At low traffic densities, the so-called \emph{free phase} in which interactions among vehicles are rare, the flux grows nearly linearly with the density until a \emph{critical density} value is reached, at which the flux takes its maximum value (\emph{road capacity}). Beyond such a critical value traffic switches to the \emph{congested phase}, which in~\cite{kerner2004BOOK} is defined as complementary to the free phase. The two phases may be separated by a \emph{capacity drop}~\cite{ZhangMultiphase}, across which the flux drops suddenly from its maximum value at free flow to a lower maximum in the congested phase. In this regime the flux decreases as the density increases. In fact interactions among vehicles are more and more frequent due to the higher packing, which causes faster vehicles to be hampered by slower ones. The formation of local slowdowns (\emph{phantom traffic jams}) is first observed. Additional increments of the density cause a steep reduction of the flux until the so-called \emph{traffic jam} is reached, in which the density reaches its maximum value $\rho_{\max}$, called \emph{jam density}, and the flux is zero.
\item[Speed-density diagrams] They give the mean speed of the vehicles as a function of the local macroscopic density of traffic. In free flow conditions, vehicles travel at the maximum allowed speed, which depends on the environmental conditions (such as e.g., quality of the road, weather conditions, infrastructure), on the mechanical characteristics of the vehicles, and on the imposed speed limits. This speed, called the \emph{free flow speed}, can be reached when there is a large distance among vehicles on the road. Conversely, in congested flow conditions vehicles travel closer to one another at a reduced speed, until the density reaches the jam density, at which vehicles stop and have zero speed.
\end{description}

These diagrams play an important role in the prediction of the capacity of a road and in the control of the flow of vehicles.

Examples of fundamental diagrams provided by experimental measurements are shown in Fig.~\ref{fig:exp_diag}. They clearly exhibit the phase transition between free and congested flow: below the critical density the flux values distribute approximately on a line with positive slope, thus the flux can be regarded as a single-valued increasing function of the density with low, though non zero, dispersion; conversely, above the critical density the flux decreases and experimental data exhibit a large scattering in the flux-density plane. In the congested phase, therefore, the flux can hardly be approximated by a single-valued function of the density, see~\cite{GreenshieldsSymposium}. Moreover, in the plot on the right a capacity drop can also be seen.

Kinetic models of traffic flow give fundamental diagrams as stationary asymptotic solutions starting from a statistical description of microscopic interactions among vehicles. In addition, some kinetic models have proved to be able to catch the transition from the free to the congested phase of traffic without building the phase transition into the model, see e.g.,~\cite{delitala2007M3AS,fermo2013SIAP}.

However, standard kinetic models do not account for the scattered data typical of the congested regime. For instance, in~\cite{HertyIllner12} multivalued fundamental diagrams are obtained supposing the flow aims to stabilize around a multivalued heuristic equilibrium velocity which is not computed by the model itself. Otherwise, this characteristic of the flow is explained considering the statistical variability of driver behaviors, who may individually decide to drive at a different speed than the one resulting from the local density, see e.g.,~\cite{fermo2014DCDSS}.

In this work we propose instead a different interpretation of the scattering of the flux in congested traffic, based on the consideration that the flow along a road is naturally \emph{heterogeneous}. Namely, it is composed by different classes of vehicles with different physical and kinematic characteristics (size, maximum speed, \dots). For this we will extend the aforementioned kinetic models so as to deal with a mixture of two populations of vehicles, say cars and trucks, each described by its own statistical distribution function. The core of the model will be the statistical description of the microscopic interactions among the vehicles of the same and of different populations, which will take into account the microscopic differences of the various types of vehicles. For the sake of simplicity, the model will be described for the case of a mixture of two populations, but it can be easily generalized to the case of several populations, see~\cite{benzoni2003EJAM}.

\section{A discrete kinetic model}
\label{sec:1pop}
In this section we briefly review the kinetic traffic model recently introduced in~\cite{fermo2013SIAP}, which will be the basis for our multipopulation extension. In the kinetic approach we focus on a statistical description of the microscopic states of the vehicles, therefore the evolution of their position $x$ and speed $v$ is described by means of a distribution function $f=f(t,\,x,\,v)$ such that $f(t,\,x,\,v)\,dx\,dv$ is the number of vehicles which at time $t$ are located between $x$ and $x+dx$ with a speed between $v$ and $v+dv$.

The model proposed in~\cite{fermo2013SIAP} is discrete both in space and in speed. Since in this work we focus on the space homogeneous case, we describe only the domain of the microscopic speeds, say $\mathcal{V}\subseteq [0,\,+\infty)$, that is:
\[
	\mathcal{V}=\{v_1,\,v_2,\,\dots,v_j,\,\dots,\,v_n\},
\]
where the $v_j$'s are \emph{speed classes} such that $0\leq v_j<v_{j+1} \quad \forall\,j=1,\,\dots,\,n-1$, and $v_1=0,\ v_n=\vmax$, $\vmax$ being the maximum speed of a vehicle. For instance, $\vmax$ can be chosen as a speed limit imposed by safety regulations, or by the state of the road, or by the mechanical characteristics of the vehicles.

The microscopic state of a generic vehicle is given by $v_j\in\mathcal{V}$ and  the statistical distribution of vehicles is given by the functions:
\[
	f_j=f_j(t):[0,\,\Tmax]\to [0,\,+\infty), \quad j=1,\,\dots,\,n,
\]
i.e. $f_j(t)$ is the number of vehicles which, at time $t$, travel with speed $v_j$.

The macroscopic variables useful in the study of traffic, namely the vehicle density $\rho$, flux $q$, and mean speed $u$ are obtained from the $f_j$'s as statistical moments with respect to the speed:
\begin{equation}
	\rho(t)=\sum_{j=1}^n f_j(t), \qquad q(t)=\sum_{j=1}^n v_jf_j(t), \qquad u(t)=\frac{q(t)}{\rho(t)}.
	\label{eq:var.macro}
\end{equation}

As already mentioned in Section~\ref{sec:fund_diag}, the experimental diagrams are measured under flow conditions which are as much as possible homogeneous in space and stationary. Thus we study the evolution in time of $f_j(t)$ due to vehicle interactions towards equilibrium. The corresponding system of (spatially homogeneous) Boltzmann-type kinetic equations writes:

\begin{equation}
	\frac{df_j}{dt}=J_j[\f,\,\f], \quad j=1,\,\dots,\,n,
	\label{eq:kinetic.1pop}
\end{equation}
where $J_j$ is the $j$-th collisional operator, which describes the microscopic interactions among vehicles causing the change of $v_j$ in time. We use the vector notation $\f:=\{f_j\}_{j=1}^{n}$. Conservation of mass requires that
\[
	\sum_{j=1}^{n}J_j[\f,\,\f]=0 \quad \forall\,\f,
\]
which ensures $\frac{d\rho}{dt}=0$.

Stationary flow conditions mean that we are actually interested in equilibrium solutions (if any) to system~\eqref{eq:kinetic.1pop}, that is constant-in-time solutions $\f^\eq=\{f^\eq_j\}_{j=1}^{n}$ such that $J_j[\f^\eq,\,\f^\eq]=0$  for all $j=1,\,\dots,\,n$. In~\cite{fermo2014DCDSS}, it is proven that $\forall \rho>0,\; \exists\,!\:\f^\eq$. Hence equilibrium solutions are parameterized by specific values of the vehicle density $\rho$, which is given by the initial condition $\rho=\sum_{j=1}^n f_j(0)$. This fact allows one to define analytically the fundamental and speed diagrams of traffic by means of the following mappings:
\[
	\rho\mapsto\f^\eq \quad \Rightarrow \quad \rho\mapsto q(\rho)=\sum_{j=1}^{n}v_jf^\eq_j, \qquad \rho\mapsto u(\rho)=\frac{q(\rho)}{\rho}.
\]
In particular, if for any given $\rho$, system~\eqref{eq:kinetic.1pop} admits a unique stable equilibrium then these mappings are indeed functions of $\rho$; otherwise, they define multivalued diagrams. We stress that, contrary to macroscopic models, the mapping $\rho\mapsto q(\rho)$ is not based on a priori closure relations but is obtained from the large time evolution of the kinetic distribution function, as a result of microscopic vehicle interactions.

\subsection{Modeling vehicle interactions}
\label{sec:1pop_interactions}
The operators $J_j$ model the microscopic interactions among vehicles. Following~\cite{fermo2013SIAP}, the formalization of $J_j$ is based on stochastic game theory. This point of view allows one to assign post-interaction speeds in a non-deterministic way, consistently with the intrinsic stochasticity of driver behaviors. We report here the construction of the operator $J_j$, which will be extended later to the two-population case. We consider only binary interactions among vehicles, thus the collisional operator can be written as:
\[
	J_j[\f,\,\f] = G_j[\f,\,\f]-f_jL_j[\f].
\]
Take:
\[
	G_j[\f,\,\f]:=\sum_{h,k=1}^n\eta_{hk} A_{hk}^jf_h f_k \qquad \text{and}
		\qquad L_j[\f]:=\sum_{k=1}^n\eta_{jk} f_k
\]
which are the \emph{gain} and \emph{loss} terms, respectively. The coefficients $\eta_{hk},\,\eta_{jk}>0$ are the \emph{interaction rates}, which depend on the relative speed of the interacting pairs: $\eta_{hk}=\eta(\abs{v_k-v_h})$ as in~\cite{coscia2007IJNM}. For simplicity, here we will assume the interaction rates independent of the pre-interaction speeds, so $\eta_{hk}\equiv \eta$ constant. 
The term $G_j$ counts statistically the number of interactions which lead, in the unit time, a so-called \emph{candidate} vehicle with speed $v_h$ to switch to the \emph{test} speed $v_j$ after an interaction with a \emph{field} vehicle with speed $v_k$. Conversely, the term $L_j$ describes the loss of vehicles with test speed $v_j$ after interactions with any field vehicle.
Thus the single-population model writes as:
\begin{equation}
	\frac{df_j}{dt}=\sum_{h,k=1}^n\eta A_{hk}^jf_h f_k-f_j\sum_{k=1}^n\eta f_k.
	\label{eq:model.1pop}
\end{equation}

For each $j=1,\dots,n$, the matrix $\A^j=\{A_{hk}^j\}_{h,k=1}^n$ is called the \emph{table of games}. It encodes the discrete probability distribution of gaining the test speed $v_j$:
\[
	A_{hk}^j=\operatorname{Prob}(v_h\rightarrow v_j\vert v_k,\,\rho), \qquad h,\,k,\,j=1,\,\dots,\,n,
\]
which in the present model is further parameterized by the macroscopic density $\rho$ so as to account for the influence of the macroscopic traffic conditions (local road congestion) on the microscopic interactions among vehicles. We stress that this is a further source of nonlinearity on the right-hand side of~\eqref{eq:kinetic.1pop}, besides the quadratic one typical of Boltzmann-like kinetic equations. Since for each fixed $j$ the coefficients $A_{hk}^j$ constitute a discrete probability distribution, they must satisfy the following conditions:
\begin{equation}
	\left.
	\begin{array}{r}
		0\leq A_{hk}^j\leq 1 \\
		\displaystyle{\sum_{j=1}^n}A_{hk}^j=1
	\end{array}
	\right\}
	\quad \forall\ h,\,k,\,j=1,\,\dots,\,n,
	\quad \forall\,\rho\in[0,\,\rhomax],
	\label{eq:prop_A}
\end{equation}
$\rhomax>0$ being the maximum density of vehicles that can be locally accomodated on the road in bumper-to-bumper conditions. These conditions ensure mass conservation.

The table of games of model~\eqref{eq:model.1pop} is built appealing to the following assumptions:
\begin{itemize}
\item a candidate vehicle with speed $v_h$ can accelerate by at most one speed class at a time. However, it can decelerate by an arbitrary number of speed classes when it interacts with a field vehicle with lower speed $v_k<v_h$;
\item let $P$ be the probability that a candidate vehicle gets the maximum possible test speed resulting from an interaction. We assume that $P$ is a decreasing function of the density $\rho$.
\end{itemize}

In more detail, we distinguish three types of interactions which determine completely the table of games. 
\begin{itemize}
\item \emph{Interaction with a faster field vehicle:} in this case we have $v_h<v_k$, or $h<k$. Following the interaction, we assume that the candidate vehicle can either accelerate or maintain its speed, thus:
\begin{equation}
	A_{hk}^j=
		\begin{cases}
			1-P & \text{if\ } j=h \\
			P & \text{if\ } j=h+1 \\
			0 & \text{otherwise}
		\end{cases}
		\qquad h<k=2,\,\dots,\,n.
	\label{h<k:1pop}
\end{equation}

\item \emph{Interaction with a slower field vehicle:} in this case we have $v_h>v_k$, or $h>k$. Following the interaction, we assume that the candidate vehicle can either maintain its speed, if for instance there is enough room to overtake the leading field vehicle, or decelerate to $v_k$ and queue up, thus:
\begin{equation}
	A_{hk}^j=
		\begin{cases}
			1-P & \text{if\ } j=k \\
			P & \text{if\ } j=h \\
			0 & \text{otherwise}
		\end{cases}
	\qquad \quad h>k=1,\,\dots,\,n-1.
	\label{h>k:1pop}
\end{equation}
Notice that in this case $P$ plays the role of a \emph{probability of passing} as defined in~\cite{prigogine1961PROC}.

\item \emph{Interaction with a field vehicle with the same speed:} in this case we have $v_h=v_k$, or $h=k$. Following the interaction, we assume that the candidate vehicle can either maintain its pre-interaction speed, or accelerate to overtake the leading vehicle, or decelerate. Hence the test speed resulting from this interaction is either $v_j=v_{h+1}$ with probability $P$, or $v_j=v_{h-1}$ with probability, say $Q$ or finally $v_j=v_h$ with probability $1-(P+Q)$. Thus $Q$ is the \emph{probability of braking} and it is chosen as an increasing function of $\rho$.

We further distinguish three cases, in fact if the candidate vehicle is either in $v_1=0$ or in $v_n=\vmax$ then it cannot decelerate or accelerate, respectively. Thus:
\begin{subequations}
\begin{equation}
	A_{11}^j=
		\begin{cases}
			1-P & \text{if\ } j=1 \\
			P & \text{if\ } j=2 \\
			0 & \text{otherwise}
		\end{cases}
	\label{1h=k:1pop}
\end{equation}

\begin{equation}
	A_{hh}^j=
		\begin{cases}
			Q & \text{if\ } j=h-1 \\
			1-(P+Q) & \text{if\ } j=h \\
			P & \text{if\ } j=h+1 \\
			0 & \text{otherwise}
		\end{cases}
	\qquad h=2,\,\dots,\,n-1
	\label{2h=k:1pop}
\end{equation}

\begin{equation}
	A_{nn}^j=
		\begin{cases}
			Q & \text{if\ } j=n-1 \\
			1-Q & \text{if\ } j=n \\
			0 & \text{otherwise}
		\end{cases}
	\label{3h=k:1pop}
\end{equation}
\end{subequations}
\end{itemize}

Note that with these choices the candidate vehicle can accelerate at most by one speed class, which amounts to bounding the maximum acceleration, see~\cite{Lebacque03}. In contrast, the deceleration is not bounded and this reflects the hypothesis that drivers behave differently in acceleration and deceleration, see the definition of \emph{traffic hysteresis} in~\cite{ZhangMultiphase} and references therein. Moreover, in the third case we use two different probabilities for the acceleration and deceleration interactions, as proposed also in~\cite{HertyIllner08}.

The choice of $P$ is crucial in our model. As in most traffic models, we will assume that accelerating is less likely in high density traffic, so $P$ is chosen as a decreasing function of $\rho$. Conversely, the probability of braking, which is either $Q$ or $1-P$, will be an increasing function of $\rho$.
Following the standard Greenshield's assumption, which is the simplest choice, we will choose, unless otherwise stated,\begin{equation}
	P=\alpha\left(1-\frac{\rho}{\rhomax}\right), \quad Q=\left(1-\alpha\right)\frac{\rho}{\rhomax}, \quad 0\leq \alpha\leq 1,
	\label{eq:probabilities}
\end{equation}
where the coefficient $\alpha\in[0,\,1]$ can be thought of as a parameter describing the environmental conditions, for instance road or weather conditions, with $\alpha=0,\,1$ standing for prohibitive and optimal conditions, respectively.

The ansatz~\eqref{eq:probabilities} is chosen by several other authors, see~\cite{HertyIllner08,klar1996,prigogine1961PROC}, but other choices are possible, see below. Further, $P$ and $Q$ can also depend on the local states, i.e. $P=P(\rho,v_h,v_k)$. This would made the model richer, but our results show that the simple choice $P=P(\rho)$ already accounts for the complexity of macroscopic data. Thus, unless otherwise stated, in the following we will take $P$ defined by~\eqref{eq:probabilities} and $\alpha=1$.

\begin{figure}[!t]
\centering
\includegraphics[width=\textwidth]{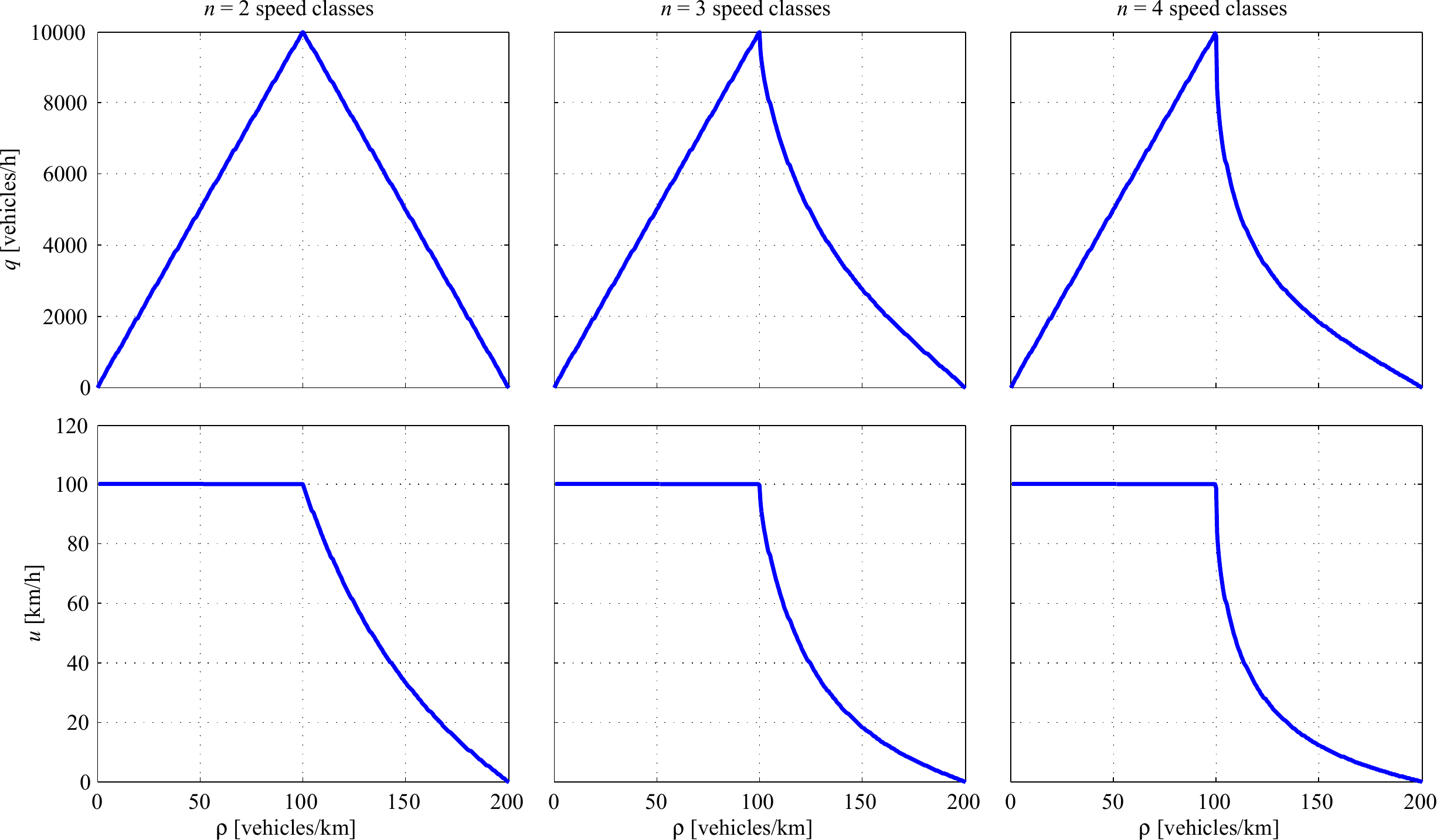}
\caption{Top row: fundamental diagrams, bottom row: speed diagrams obtained from model~\eqref{eq:model.1pop} with $n=2,\,3,\,4$ speed classes, maximum density $\rhomax=200~\unit{vehicles/km}$, and uniformly distributed microscopic speeds in the interval $[0,\,100~\unit{km/h}]$.}
\label{fig:tosin}
\end{figure}

By computing the evolution towards equilibrium of a given initial condition corresponding to a fixed value of $\rho$ we obtain the fundamental and speed diagrams depicted in Fig.~\ref{fig:tosin} for three different values of the number $n$ of speed classes. For $\rho<\frac{1}{2}\rhomax$ we recognize the free phase of traffic, in which the flux is an increasing linear function of the density. Conversely, for $\rho>\frac{1}{2}\rhomax$ we find the congested phase, in which the specific form of the diagrams predicted by the model depends on the number $n$ of speed classes. Note that for $n=3$ and $n=4$ the phenomenon of capacity drop becomes apparent. 

The value of $\rho_c=\frac{1}{2}\rhomax$ is due to the particular choice of $P$. If we take instead
\begin{equation}\label{eq:gamma_law}
	P=\alpha\left(1-\left(\frac{\rho}{\rhomax}\right)^{\gamma}\right), \quad \gamma>0
\end{equation}
then the critical density $\rho_c$ decreases when $\gamma<1$, see Figure~\ref{fig:square} in Section~\ref{sec:fundamental}, and the discussion in \S \ref{sec:equilibria}. The value $\rho_c$ mentioned in Section~\ref{sec:fund_diag}, at which a bifurcation of equilibria occurs, is the mathematical counterpart of the physical phase transition. See also~\cite{fermo2014DCDSS}.

These results confirm that the kinetic approach is able to catch successfully the phase transition in traffic flow as a consequence of more elementary microscopic interaction rules. In particular, such a phase transition need not be postulated a priori through heuristic closures of the flux as a given function of the density. Nevertheless, model~\eqref{eq:model.1pop} still provides a single-valued density-flux relationship. In fact, as shown in~\cite{fermo2014DCDSS}, for all $\rho\in[0,\,\rhomax]$ there exists a unique stable and attractive equilibrium $\f^\eq=\{f_j^\eq\}_{j=1}^n$. Consequently, the flux $q$ at equilibrium is uniquely determined by the initial density $\rho$, which does not explain the scattered data of the experimental diagrams.

\section{Two-population models}
\label{sec:2pop}
Starting from the kinetic approach discussed in the previous section, we now introduce a model which treats traffic as a mixture of different types of vehicles with different physical and kinematic characteristics. As far as we know, this is the first attempt to account for the heterogeneity of traffic in a kinetic model. We will see that the proposed structure allows one to account for the nature of scattered data in experimental diagrams. For the sake of simplicity we will consider a two-population model, which can be easily extended to more complex mixtures.

Multi-population models of vehicular traffic are already available in the literature, e.g. see~\cite{benzoni2003EJAM,LebacqueGSOM,MendezVelasco13}. Here we start from~\cite{benzoni2003EJAM}, in which the authors describe an $M$-population generalization of the Lighthill-Whitham-Richards macroscopic traffic models~\cite{lighthill1955PRSL,richards1956OR}, that we briefly illustrate in the case of $M=2$ species as an introduction to the forthcoming kinetic approach.

Let $N_\p(t,x)$ be the number of vehicles of the $\p$-th population, $\p=1,2$, contained in a stretch of road of length $L$ (typically $L$ will be 1 kilometer). The model consists of two coupled one-dimensional conservation laws:
\begin{equation}
	\begin{cases}
		\partial_t\rho_1+\partial_x F_1(\rho_1,\,\rho_2)=0 \\
		\partial_t\rho_2+\partial_x F_2(\rho_1,\,\rho_2)=0
	\end{cases}
\label{mod:colombo}
\end{equation}
where $\rho_\p=\rho_\p(t,\,x)=N_\p(t,x)/L$ is the macroscopic density of the $\p$-th species, $F_\p(\rho_1,\,\rho_2)=\rho_\p v_\p(\rho_1,\,\rho_2)$ its flux function, and $v_\p(\rho_1,\,\rho_2)$ is the speed-density relation, which describes the attitude of drivers of the $\p$-th population to change speed on the basis of the local values of $\rho_1$, $\rho_2$. The model is based on the idea that the $\p$-th population is characterised by vehicles with length $l_\p>0$ and maximum velocity $V_\p>0$. Then, one can define the \emph{fraction of road occupancy} as the dimensionless quantity
\[ 
	s:=\rho_1l_1+\rho_2l_2, \qquad 0\leq s \leq 1,
\]
and consider the following extension of the Greenshield speed-density relation:
\[
	v_\p(\rho_1,\,\rho_2)=\left(1-s\right)V_\p, \quad \p=1,\,2.
\]
With these choices, system~\eqref{mod:colombo} becomes\begin{equation}
	\begin{cases}
		\partial_t\rho_1+\partial_x(\rho_1(1-s)V_1)=0 \\
		\partial_t\rho_2+\partial_x(\rho_2(1-s)V_2)=0
	\end{cases}
\label{mod:colombo2}
\end{equation}
with fluxes $F_\p(\rho_1,\,\rho_2)=\rho_\p(1-s)V_p$. We notice that the total flux $F_1+F_2=(1-s)(\rho_1V_1+\rho_2V_2)$ is not a one-to-one function of the fraction of road occupancy $s$, as there might exist different pairs $(\rho_1,\,\rho_2)$ giving rise to the same value of $s$ and nevertheless to different total fluxes. This is possible provided $l_1\ne l_2$ or $V_1\ne V_2$.

\subsection{A two-population kinetic model}
\label{sec:2pop.kin}
In constructing our two-population kinetic model we confine ourselves to the spatially homogeneous case, in order to focus on the study of fundamental diagrams. To fix ideas, we identify the two classes of vehicles with ``cars'' ($C$) and ``trucks'' ($T$), respectively. Roughly speaking, the physical and kinematic differences between them consist in that cars are shorter and faster than trucks, therefore $l^C\leq l^T$ and $V^C\geq V^T$. Clearly, other choices are possible, see Section~\ref{sec:fundamental}. We adopt a compact notation, which makes use of two indices:
\[
	\p\in\{C,\,T\}, \qquad \q = \neg\, \p
\]
to label various quantities referred to either population of vehicles.

We assume that the discrete spaces of microscopic speeds for cars and trucks, $\mathcal{V}^C$, $\mathcal{V}^T$, respectively, are such that $\mathcal{V}^T\subseteq\mathcal{V}^C$, i.e. the speeds accessible to trucks are a subset of those accessible to cars.  For simplicity, we take $\mathcal{V}^C$ as an equispaced lattice of speeds, i.e.
\[
	\mathcal{V}^C =\left\{ v_j=\frac{j-1}{n^C-1}\vmax, \quad 1\leq j\leq n^C \right\},
\]
where $n^C$ is the number of speed classes for cars, then let $\mathcal{V}^T=\{v_j\}_{j=1}^{n^T}$ with $n^T\leq n^C$. This way the maximum speed of cars is $V_C=\vmax$, whereas the maximum speed of trucks is $V_T=\frac{n^T-1}{n^C-1}\vmax\leq \vmax$.

On the discrete space $\mathcal{V}^\p$ we introduce the kinetic distribution function
\[
	f^\p_j=f^\p_j(t):[0,\,\Tmax]\to [0,\,+\infty), \quad \p\in\{C,\,T\}, \quad j=1,\dots, n^{\p},
\]
which gives the statistical distribution of $\p$-vehicles traveling with speed $v_j$ at time $t$. The macroscopic observable quantities referred to such a class of vehicles are recovered as (cf.~\eqref{eq:var.macro}):
\begin{equation}
	\rho^\p(t)=\sum_{j=1}^{n^\p}f^\p_j(t), \qquad
	q^\p(t)=\sum_{j=1}^{n^\p}v_jf^\p_j(t), \qquad
	u^\p(t)=\frac{q^\p(t)}{\rho^\p(t)}.
	\label{eq:var.macro.p}
\end{equation}

We model the evolution of the $f^\p_j$'s by means of the following equation
\begin{equation}
	\frac{df^\p_j}{dt}=J^\p_j\left[\f^\p,\,\left(\f^\p,\,\f^\q\right)\right], \quad j=1,\,\dots,\,n^\p
	\label{eq:kinetic.2pop}
\end{equation}
where the term $J^\p_j\left[\f^\p,\,\left(\f^\p,\,\f^\q\right)\right]$ describes the interactions of $\p$-vehicles with all other vehicles, in which, as a result, the $\p$ vehicle assumes the velocity $v_j$. Since we consider only binary interactions, we can follow an approach frequently used for mixtures of two gases in kinetic theory, see e.g.,~\cite{brull2012EJMB,groppi2004PF,hamel1965PF}, which consists in writing the collisional operator as the sum of two terms:
\begin{equation}
	J^\p_j\left[\f^\p,\,\left(\f^\p,\,\f^\q\right)\right]=J^{\p\p}_j[\f^\p,\,\f^\p]+J^{\p\q}_j[\f^\p,\,\f^\q], \quad j=1,\,\dots,\,n^\p.
	\label{eq:J_2pop}
\end{equation}
In particular, the term $J^{\p\p}_j[\f^\p,\,\f^\p]$ accounts for \emph{self-interactions} within the population $\p$, i.e., interactions in which $\p$-vehicles play also the role of field vehicles. Conversely, the term $J^{\p\q}_j[\f^\p,\,\f^\q]$ accounts for \emph{cross-interactions} between the two populations. 
Following the same logic underlying the single population model, cf. Section~\ref{sec:1pop_interactions}, each term is written as a balance of gain and loss contributions:
\begin{equation}
	\begin{array}{l}
		J^{\p\p}_j[\f^\p,\,\f^\p]=\ds{\sum_{h,k=1}^{n^\p}}\eta_{hk}^\p A^{\p,j}_{hk}f^\p_hf^\p_k
			-f^\p_j\ds{\sum_{k=1}^{n^\p}}\eta_{jk}^\p f^\p_k \\[6mm]
		J^{\p\q}_j[\f^\p,\,\f^\q]=\ds{\sum_{h=1}^{n^\p}}\ds{\sum_{k=1}^{n^\q}}\eta_{hk}^{\p\q} B^{\p\q,j}_{hk}f^\p_hf^\q_k
			-f^\p_j\ds{\sum_{k=1}^{n^\q}}\eta_{jk}^{\p\q} f^\q_k
	\end{array}
	\qquad j=1,\,\dots,\,n^\p,
	\label{eq:self_cross-2pop}
\end{equation}
where $\A^{\p,j}$, $\B^{\p\q,j}$, $j=1,\,\dots,\,n^\p$, are the self-interaction and cross-interaction tables of games, respectively. Since the coefficients of the two tables of games model the transition probabilities, we require that
\begin{gather*}
	0\leq A^{\p,j}_{hk},\, B^{\p\q,j}_{hk}\leq 1, \quad \forall\,h,\,k,\,j,\,\p,\,\q \\
	\sum_{j=1}^{n^\p}A^{\p,j}_{hk}=\sum_{j=1}^{n^\p}B^{\p\q,j}_{hk}=1, \quad \forall\, h,\,k,\,\p,\,\q
\end{gather*}
so that for each $h,\,k$ fixed, the coefficients $A^{\p,j}_{hk}$ and $B^{\p\q,j}_{hk}$ with $j=1,\dots,\p$ form indeed discrete probability densities. This ensures that:
\[
	\sum_{j=1}^{n^\p}J^\p_j[\f^\p,\,\f^\q]=\sum_{j=1}^{n^\p}J^{\p\p}_j[\f^\p,\,\f^\p]
		+\sum_{j=1}^{n^\p}J^{\p\q}_j[\f^\p,\,\f^\q]=0,
\]
whence from~\eqref{eq:kinetic.2pop} mass conservation for each species is obtained:
\[
	\frac{d}{dt}\sum_{j=1}^{n^\p}f^\p_j=\frac{d\rho^\p}{dt}=0.
\]

Finally, the two-population model resulting from~\eqref{eq:kinetic.2pop}--\eqref{eq:self_cross-2pop} can be written as
\begin{equation}
	\frac{df^\p_j}{dt}=\sum_{h,k=1}^{n^\p}\eta_{hk}^\p A^{\p,j}_{hk}f^\p_hf^\p_k+\sum_{h=1}^{n^\p}\sum_{k=1}^{n^\q}\eta_{hk}^{\p\q} B^{\p\q,j}_{hk}f^\p_hf^\q_k
		-f^\p_j\left(\sum_{k=1}^{n^\p}\eta_{jk}^\p f^\p_k+\sum_{k=1}^{n^\q}\eta_{jk}^{\p\q} f^\q_k\right), \quad j=1,\,\dots,\,n^\p.
	\label{eq:model.2pop}
\end{equation}
Here, the {\em interaction rates} $\eta_{hk}^\p,\,\eta_{hk}^{\p\q}$ may  depend on the type of interacting vehicles and on the relative speeds between the vehicles, but, for the sake of simplicity, in the following we will assume that they are constant, let $\eta=\eta_{hk}^\p=\eta_{hk}^{\p\q}$. In particular, in all numerical tests, since we are interested in equilibrium solutions, $\eta$ will be taken equal to $1$ without loss of generality.


\begin{remark}
In~\cite{fermo2014DCDSS}, it is proven that for a single population, if $f_j(0)\geq 0\,\forall j$, then $f_j(t)\geq 0$ for all times, and the equilibrium distribution $\f^\eq$ is uniquely determined by the initial condition. Here, our numerical evidence suggests that the same properties are inherited also by the multipopulation model when the well balanced scheme of Section~\ref{sec:well-balanced} is used. This fact can also be proven, see~\cite{puppoUNPUBL}.
\end{remark}

\subsubsection{Modeling self- and cross-interactions}
The total number of $\p$-vehicles present in a stretch of road of length $L>0$ is $N^\p=L\sum_{j=1}^{n^\p}f^\p_j$. Recall that $l^\p>0$ is the characteristic length of $\p$-vehicles, therefore the total space occupied by population $\p$ along the road is $N^\p l^\p$, while the total space occupied by all vehicles is $S=\sum_{\p\in\{C,\,T\}}N^\p l^\p$. Ultimately, the \emph{fraction of road occupancy} over the length $L$ is
\begin{equation}
	s:=\frac{S}{L}=\sum_{\p\in\{C,\,T\}}\frac{N^\p}{L}l^\p=
		\sum_{\p\in\{C,\,T\}}\left(\sum_{j=1}^{n^\p}f^\p_j\right)l^\p=\sum_{\p\in\{C,\,T\}}\rho^\p l^\p.
	\label{eq:s}
\end{equation}
Let $\rhomax^\p$ be the maximum density of vehicles of the $\p$-th population, which is obtained when the road is completely filled and  $\rho^\q=0$. Obviously, given $l^C$, $l^T$, the admissible pairs of densities $(\rho^C,\,\rho^T)\in [0,\,\rhomax^C]\times [0,\,\rhomax^T]$ are those such that $0\leq s\leq 1$.

Notice that $\rhomax^\p=\frac{1}{l^\p}$, therefore $s$ can be rewritten as
\[
	s=\sum_{\p\in\{C,\,T\}}\frac{\rho^\p}{\rhomax^\p}.
\]
From this expression it is clear that $s$ is the natural generalization of the term $\frac{\rho}{\rhomax}$ appearing in the probabilities $P$, $Q$ of the single-population model, cf. Section~\ref{sec:1pop_interactions}. Therefore we will assume that in the two-population model the transition probabilities depend on $s$. In other words, following the same logic of the single population case, the elements of the table of games depend on the local state of occupancy of the road, which, when more than one population is present, is given by $s$.  More precisely, $P$ is a decreasing function of $s$, while $Q$ is an increasing function of $s$. Following~\eqref{eq:probabilities}, the simplest choice is 
\begin{equation}
	P=\alpha(1-s), \qquad Q=(1-\alpha)s.
	\label{eq:probab_s}
\end{equation}
Other choices are possible, as in the case of  the $\gamma$-law \eqref{eq:gamma_law}, see \S \ref{sec:equilibria} and Fig. \ref{fig:square}. It would also be possible to consider different reactive behaviours in the two populations. But the simplest choice, which, as we will see, results in a realistic macroscopic behaviour, is to suppose that both types of vehicles react in the same way to the single parameter which accounts for the state of occupation of the road, which is $s$. The result of this ansatz is that the tables of games differ only in their dimensions.
  
For the matrices $\A^{\p,j}$ we use the same construction as in~\eqref{h<k:1pop}--\eqref{3h=k:1pop}, because they express self-interactions within either population of vehicles regardless of the presence of the other population. The only difference is that they are $n^\p\times n^\p$ matrices, hence their dimensions change depending on the specific population.

The tables of games $\B^{\p\q,j}$ are instead $n^\p\times n^\q$ rectangular matrices, therefore we need to slightly revise the basic interaction rules of the single-population model in order to take into account the different maximum speeds of the two populations in the description of the speed transitions.

The table $\B^{CT,j}$ gives the probability distribution that candidate cars switch to the test speed $v_j$ upon interacting with field trucks. The coefficients $B^{CT,j}_{hk}$ are constructed as in~\eqref{h<k:1pop}--\eqref{3h=k:1pop}, considering however that the case~\eqref{2h=k:1pop} applies only for $h=2,\,\dots,\,n^T\leq n^C$ and that the case~\eqref{3h=k:1pop} applies only if $n^T=n^C$. Thus the matrices $\B^{CT,j}$ are $n^T\times n^C$.

Conversely, the table $\B^{TC,j}$ gives the probability distribution that candidate trucks switch to the test speed $v_j$ upon interacting with field cars. If the candidate truck is faster than the field car then the coefficient $B^{TC,j}_{hk}$ is constructed as in~\eqref{h>k:1pop}. Instead, when interactions involve also accelerations it is necessary to consider that the candidate truck might not be able to increase its speed if it is already traveling at its maximum possible velocity $v_{n^T}$, which, unless $n^T=n^C$, is in general smaller than $\vmax$. In other words, in the case~\eqref{h<k:1pop} the option of accelerating may not apply. Hence for candidate trucks traveling at speed $v_{n^T}$ which encounter faster field cars, i.e., cars traveling at speed $v_k$ with $k=n^T+1,\,\dots,\,n^C$, we modify the transition probabilities~\eqref{h<k:1pop} as:
\[
	B^{TC,j}_{n^Tk}=
		\begin{cases}
			1 & \text{if\ } j=n^T \\
			0 & \text{otherwise},
		\end{cases}
		\qquad k>n^T.
\]
Notice instead that the other cases which include an acceleration, namely~\eqref{1h=k:1pop} and~\eqref{2h=k:1pop}, can be borrowed from the single-population model without modifications.

\medskip

Interestingly, model~\eqref{eq:model.2pop} along with the tables of games discussed above satisfies an \emph{indifferentiability principle} similar to the one valid for kinetic models of gas mixtures, see e.g.,~\cite{andries2001REPORT}: when all the species composing the gas are identical one recovers the equations of a single-component gas. In the present case, the indifferentiability principle can be stated as follows:
\begin{theorem}[Indifferentiability principle]
Assume that the two types of vehicles are identical, i.e., they have the same physical and kinematic characteristics, and $f_j^{\p}$ exists $\forall\,\p\in\{C,T\}$, $\forall\,j$. Then the total distribution function
\begin{equation} 
	f_j:=\sum_{\p\in\{C,\,T\}}f^\p_j
	\label{eq:definit}
\end{equation}
obeys the evolution equations of the single-population model~\eqref{eq:model.1pop}.
\label{theo:indiffpr}
\end{theorem}
\begin{proof}
If the two types of vehicles are the same we have $l^C=l^T=:l$, $\rhomax^C=\rhomax^T=\frac{1}{l}=:\rhomax$. It follows
\[
	s=\frac{\rho^C+\rho^T}{\rhomax}=\frac{\rho}{\rhomax}.
\]

Since $\mathcal{V}^T \subseteq \mathcal{V}^C$, if the maximum speed is the same then $n^C=n^T:=n$ and $\mathcal{V}^T = \mathcal{V}^C$. This implies $\A^{\p,j}=\B^{\p\q,j}=\A^j$, the latter being the table of games of the single-population model, cf. Section~\ref{sec:1pop_interactions}. Further, since the two populations are identical, the interaction rates are the same. So taking these facts into account and summing~\eqref{eq:model.2pop} over $\p$ yields:
\[
	\frac{d}{dt}\sum_{\p\in\{C,\,T\}}f^\p_j=\sum_{h,k=1}^{n}\eta_{hk} A^j_{hk}\left(\sum_{\p\in\{C,\,T\}}f^\p_h\right)\left(f^\p_k+f^\q_k\right)
		-\left(\sum_{\p\in\{C,\,T\}}f^\p_j\right)\sum_{k=1}^{n}\eta_{jk}\left(f^\p_k+f^\q_k\right),
\]
whence, using the definition~\eqref{eq:definit}, we have:
\[
	\frac{df_j}{dt}=\sum_{h,k=1}^{n}\eta_{hk}A^j_{hk}f_hf_k-f_j\sum_{k=1}^{n}\eta_{jk}f_k,
\]
which concludes the proof.
\end{proof}

\begin{remark}
In~\cite{andries2001REPORT} the indifferentiability principle is proved for a model featuring a single collision operator, which hinders the description of cross-interactions among particles of different species in the mixture. In more standard models for gas mixtures the collision terms are separate, as in our case, but the indifferentiability principle holds only at equilibrium. Here, instead, Theorem~\ref{theo:indiffpr} holds at all times, moreover without having to merge the two collision terms into one.
\end{remark}

\subsection{A well-balanced formulation for computing equilibria}
\label{sec:well-balanced}
Our numerical evidence suggests that, for any pair of densities $(\rho^C,\,\rho^T)\in[0,\,\rhomax^C]\times [0,\,\rhomax^T]$, with  $0\leq s \leq 1$, all initial distributions $(\f^C(0)\geq 0,\,\f^T(0)\geq 0)$ such that $\sum_{j=1}^{n^\p}f^\p_j(0)=\rho^\p$, $\p\in\{C,\,T\}$, converge in time to the same pair of equilibrium distributions $(\f^{\eq,C},\,\f^{\eq,T})$, which is therefore uniquely determined by $\rho^C$ and $\rho^T$. The proof of this and other analytical properties can be found in~\cite{puppoUNPUBL}. To get the correct equilibrium, however, it is important to devise a well balanced numerical scheme. As we will see, round-off error can drive the solution to spurious equilibrium states, if the model is not integrated properly.

Under the simplifying assumption $\alpha=1$, in~\cite{fermo2014DCDSS} existence and uniqueness of stable equilibria for the single-population model~\eqref{eq:model.1pop} are established and their analytic expressions are computed. In the general case, it is necessary to integrate numerically in time the system of ODEs~\eqref{eq:model.2pop} until steady state is reached.

It is worth pointing out that in~\cite{fermo2014DCDSS} equilibria are studied by rewriting the loss term $-f_j\sum_{k=1}^{n}f_k$ of~\eqref{eq:model.1pop} in the analytically equivalent form $-\rho f_j$. This allows one to take advantage of the fact that $\rho$ is indeed a parameter of system~\eqref{eq:model.1pop} fixed by the initial condition, since it is constant in time. However, such a simplification cannot be carried out when the system is integrated numerically, because of instabilities triggered by round-off errors. 

For the sake of simplicity, we illustrate this phenomenon for the single-population model~\eqref{eq:model.1pop} but our considerations apply to the two-population model~\eqref{eq:model.2pop} as well.
For this purpose, we consider the numerical approximation of the following two analytically equivalent formulations of the single-population model:
\begin{subequations}
\begin{gather}
	\frac{df_j}{dt}=\sum_{h,k=1}^n A_{hk}^jf_hf_k-f_j\sum_{k=1}^nf_k \label{eq:well-balanced} \\
	\frac{df_j}{dt}=\sum_{h,k=1}^n A_{hk}^jf_hf_k-\rho f_j \label{eq:not_well-balanced},
\end{gather}
\end{subequations}
where the first, which we will call \emph{well-balanced}, leads to the computation of the correct equilibria while the second does not preserve stationary solutions and possibly leads to a violation of mass conservation. The context is similar to the construction of well-balanced numerical schemes for balance laws, where particular care is needed in order to preserve stationary solutions at the discrete level, see e.g.,~\cite{leveque1998JCP,noelleUNPUBL} and references therein.

\begin{figure}[!t]
\centering
\includegraphics[width=\textwidth]{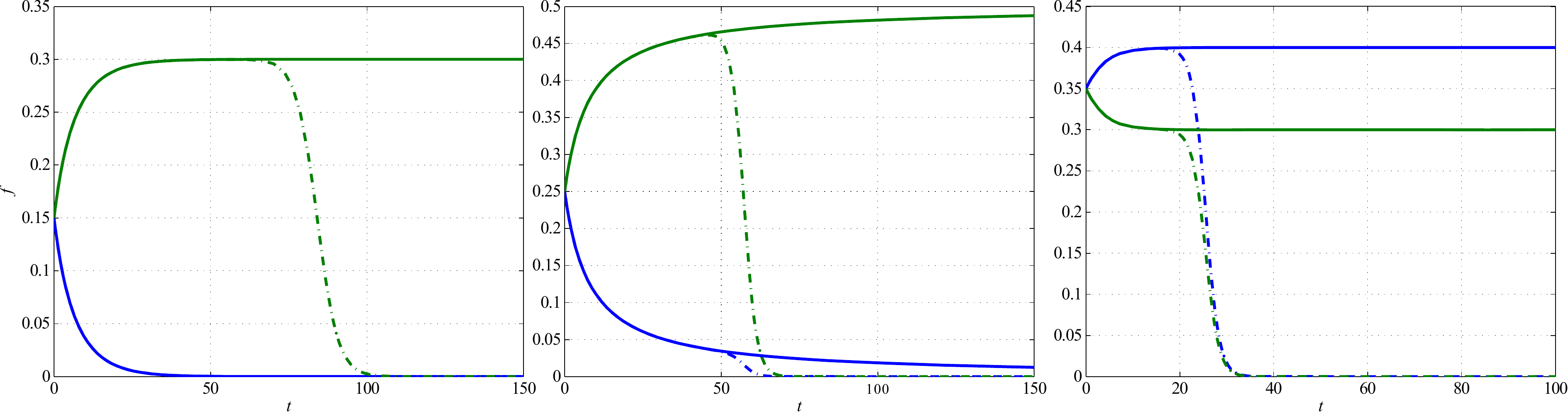}
\caption{Solution to the spatially homogeneous problem for the single-population model obtained with (continuous line) and without (dashed line) the well-balanced formulation~\eqref{eq:well-balanced}. Left $\rho=0.3$, center $\rho=0.5$, right $\rho=0.7$.}
\label{fig:WB}
\end{figure}

Let $y(t)=\sum_{j=1}^{n}f_j(t)$. Summing over $j$ both sides of~\eqref{eq:well-balanced} yields $\frac{dy}{dt}=0$ as expected, while the same operation performed on~\eqref{eq:not_well-balanced} gives
\[
	\frac{dy}{dt}=(y-\rho)y.
\]
In this case $y$ is in general not constant in time and moreover two equilibria exist, $y^\eq_1=0$ and $y^\eq_2=\rho$, where the first is stable and attractive, whereas the second is unstable. This means that mass conservation $y(t)=\rho$ holds for all $t$ if and only if $y(0)=\rho$ and $y$ is computed without round-off. Otherwise, any small perturbation will drive $y$ away from the unstable equilibrium $y^\eq_2=\rho$ towards the stable equilibrium $y^\eq_1=0$ and mass conservation fails.

Figure~\ref{fig:WB} shows the results of the numerical integration of the two equations~\eqref{eq:well-balanced},~\eqref{eq:not_well-balanced} in the simple case with $n=2$ speed classes and $\alpha=1$, starting from initial conditions for which the density is $\rho=0.3$, $\rho=0.5$, or $\rho=0.7$ respectively. As proved in~\cite{fermo2014DCDSS}, the correct equilibrium distribution is $\f^\eq=(0,\,\rho)$ for $\rho\leq 0.5$, but from the first two panels of Fig.~\ref{fig:WB} it can be seen that only the solution of~\eqref{eq:well-balanced} converges to such an equilibrium, while~\eqref{eq:not_well-balanced} is attracted toward the state $(0,\,0)$ which violates mass conservation. For $0.5<\rho<1$ the correct equilibrium distribution $\f^\eq$ consists instead of two strictly positive values, which once again are reached only by the numerical solution of the well-balanced formulation~\eqref{eq:well-balanced} as it is evident from the third panel of Fig.~\ref{fig:WB}.



\subsection{Transition from free to congested phase}
\label{sec:equilibria}
In this section, we compute the value of $s$ which determines the transition from the free to the congested phase. For this purpose, we will compute 
the equilibria of the system~\eqref{eq:model.2pop} in the free flow phase. Our goal is to investigate analytically the main characteristics of the fundamental diagrams resulting from our model. In particular, using the equilibria of the sum of the two distribution functions, we compute the value of occupied space, the  \emph{critical space} $s_c$, at which the transition from the free to the congested phase occurs, and we will see how $s_c$ depends on the choice of the probability $P$. 

In order to compute the equilibria, we need the explicit expression of the interaction matrices.
We will write the table of games explicitly for the $\gamma$-law \eqref{eq:gamma_law}, with $\alpha=1$. Thus we have
\[
P=1-s^{\gamma}, \qquad Q=0.
\]
As a result, the structure of the tables of games is considerably simplified, nevertheless the wealth of information which can be extracted from the model is still surprising. Note also the sparsity pattern of the matrices $\A^{\p,j}$, $\B^{\p\q,j}$, which permits a fast evaluation of the collision terms in~\eqref{eq:kinetic.2pop}. 

Let $R:=1-P$. We report only the non zero elements, drawing a circle around the elements which belong to the $j$-th row and column of each of the interaction matrices, $\A^{\p,j}$ and $\B^{\p\q,j}$ respectively. Inside the circle we indicate the value of the corresponding element.

Concerning the self-interaction table of games we have:
\[
	\A^{\p,1}=
		\begin{bmatrix}
			\circled{R} & \circled{R}& \circled{R} & \cdots & \circled{R} \\
			\circled{R} \\
			\circled{R} \\
			\vdots \\
			\circled{R}         
		\end{bmatrix},
	\qquad
	\A^{\p,n^\p}=
		\begin{bmatrix}	
			& & & & \circled{0} \\
			& & & & \circled{0} \\
			& & & & \vdots \\
			& & & P & \circled{P} \\
			\circled{P} & \cdots & \circled{P} & \circled{P} & \circled{1}
		\end{bmatrix},
\]
while the general expression for $1<j<n^\p$ is
\[
	\A^{\p,j}=
		\begin{bmatrix}
			& & & \circled{0} & & \\
			& & & \vdots & & \\
			& & P & \circled{P} & \cdots & P \\
			\circled{P} & \cdots & \circled{P} & \circled{R} & \cdots & \circled{R} \\
			& & & \circled{R} \\
			& & & \vdots \\
			& & & \circled{R}
		\end{bmatrix}.
\]
These matrices are all $n^\p\times n^\p$.
The cross-interaction matrices $\B^{CT,j}$ between cars (candidates) and trucks (fields) have the same structure as the $\A^{\p,j}$'s, apart from being rectangular of dimensions $n^C\times n^T$. Differences however arise for $j\geq n^T$, for example:
\[
	\B^{CT,n^T}=
		\begin{bmatrix}
			& & & \circled{0} \\
			& & & \vdots \\
			& & P & \circled{P} \\
			\circled{P} & \cdots & \circled{P} & \circled{R} \\
			& & & \circled{R} \\
			& & & \vdots \\
			& & & \circled{R}
		\end{bmatrix},
	\qquad
	\B^{CT,n^T+1}=
		\begin{bmatrix}
			\\
			\\
			\\
			& & & P \\
			\circled{P} & \cdots & \cdots & \circled{P} \\
			\\
			\\
			\\
		\end{bmatrix}
\]
and in general
\[
	\B^{CT,j}=
		\begin{bmatrix}
			\\
			\\
			\\
			\\
			\\
			\circled{P} & \cdots & \cdots & \circled{P} \\
			\\
			\\
		\end{bmatrix}, \quad j>n^T+1.
\]

Finally, the cross-interaction matrices $\B^{TC,j}$ between trucks (candidates) and cars (fields) are $n^T\times n^C$. They can in turn be easily derived from the $\A^{\p,j}$'s, the only different case being the one for $j=n^T$:
\[
	\B^{TC,n^T}=
		\begin{bmatrix}
			& & & \circled{0} \\
			& & & \vdots \\
			& & P & \circled{P} & \cdots & P \\
			\circled{P} & \cdots & \circled{P} & \circled{1} & \circled{1} & \circled{1}      
		\end{bmatrix}.
\]

\bigskip

We will assume that the distribution functions are non negative in time provided $f^{\p}_j(0)\geq 0, \,\forall j$. Our numerical evidence supports this assumption, but see also~\cite{puppoUNPUBL}. Let $F_j=\sum_{\p}f^{\p}_j$ and $\rho=\sum_{\p} \rho^{\p}$.

Summing the equations $\frac{df_1^{\p}}{dt}=0$ and $\frac{df_1^{\q}}{dt}=0$, we have:
\begin{equation}
	-R \left(F_1\right)^2+(2R-1)\rho F_1=0
	\label{eq:1}
\end{equation}
which is a quadratic equation whose non-negative roots are:
\[
	F_1=
	\begin{cases}
		0\\
		\frac{(2R-1)\rho}{R}, \quad \text{if $R\geq\frac12$}.
	\end{cases}
\]
Since the leading coefficient of the equation~\eqref{eq:1} is negative, the stable and attractive equilibrium is always the largest root: thus, for $R\leq 1/2$, the equilibrium solution is $F_1=0$. Since the two  distribution functions are non negative, the equilibrium for each population is $f^{\p}_1=0$, i.e. no vehicle travels with velocity $v_1=0$.

Henceforth, we take $R\leq 1/2$ and we suppose $f^{\p}_{j-1}=0, \p=C,T$, $\forall\,j<n^T$ (inductive hypothesis) and we prove that $f^{\p}_j=0$. By summing again the equations $\frac{df^{\p}_j}{dt}=0$, we obtain:
\[
	-R\left(F_j\right)^2+F_j\left[(1-3R)\sum_{k=1}^{j-1}F_k+(2R-1)\rho\right]+(1-R)F_{j-1}\left[\rho-\sum_{k=1}^{j-2}F_k\right]=0
\]
and using the inductive step, this expression reduces to
\[
	-R\left(F_j\right)^2+F_j(2R-1)\rho=0
\]
which again has a stable root at $F_j=0$ for $R\leq 1/2$, closing the induction.
Therefore, if $R\leq 1/2$, the stable and attractive equilibrium of each species is $f^{\p}_j=0$, $\forall\, j=1,\dots,n^T-1$. For $j=n^T$, using mass conservation for trucks we have
\[
	\rho^T = \sum_{j=1}^{n^T} f^T_j = f^T_{n^T}.
\]
Thus, for $R\leq 1/2$, all trucks travel with the maximum velocity allowed in $\mathcal{V}^T$, at 
equilibrium. Using this result, the remaining equations for $f^C_j, j=n^T,\dots,n^C$ can be written once $\f^T$ is known. 
The equilibrium related to the distribution function of cars traveling at the velocity $v_{n^T}$ can be found by solving the quadratic equation of $f_{n^T}^C$ resulting from $\frac{d}{dt}f_{n^T}^C=0$:
\[
	-R\left(f_{n^T}^C\right)^2+f_{n^T}^C\left[(2R-1)\rho^C-\rho^T\right]+R\rho^C\rho^T=0.
\]
Since
\[
	\Delta_{n^T}=\left[(2R-1)\rho^C-\rho^T\right]^2+4R^2\rho^C\rho^T
\]
is non-negative $\forall\,R\in[0,1]$, all solutions are real. In particular, clearly they have opposite sign if $\rho_C$ and $\rho_T$ are non zero, and the positive root is
\[
	f_{n^T}^C=\frac{(2R-1)\rho^C-\rho^T+\sqrt{\Delta_{n^T}}}{2R}.
\]
This root represents the stable and attractive equilibrium. This result depends implicitly on the assumption $R\leq 1/2$ because it exploits the fact that $\f^{\eq,T}=[0,\dots,0,\rho_T]$ which holds if $R\leq 1/2$. For $j=n^T+1,\dots,n^C-1$, equilibrium distributions of cars result from the equations
\[
	-R\left(f_j^C\right)^2+f_j^C\left[(1-3R)\sum_{k=n^T}^{j-1} f_k^C + (2R-1)\rho^C-R\rho^T\right]+(1-R)c_j=0,\quad j=n^T+1,\dots,n^C-1,
\]
where the coefficient $c_j$ is
\[
	c_j=
	\begin{cases}
	f^C_{n^T}\rho,	&\text{if\, $j=n^T+1$}\\
	\ds{f_{j-1}^C\left(\rho^C-\sum_{k=n^T}^{j-2} f_k^C\right)}, &\text{if\, $j=n^T+2,\dots,n^C-1$.}
	\end{cases}
\]
Again, the roots of the quadratic equation are real and of opposite sign. The largest one is
\[
f_j^C=\frac{\ds{(1-3R)\sum_{k=n^T}^{j-1} f_k^C + (2R-1)\rho^C-R\rho^T+\sqrt{\Delta_j}}}{2R}, \quad j=n^T+1,\dots,n^C-1
\]
and it is the stable and attractive equilibrium, where
\[
	\Delta_j=\left[(1-3R)\sum_{k=n^T}^{j-1} f_k^C + (2R-1)\rho^C-R\rho^T\right]^2+4R(1-R)c_j, \quad j=n^T+1,\dots,n^C-1
\]
is the non-negative discriminant of the equation for $f_j^C$, $j=n^T+1,\dots,n^C-1$.

Finally, by mass conservation, the asymptotic distribution related to cars traveling at the maximum velocity $v_{n^C}$ is
\[
	f_{n^C}^C=\rho^C-\sum_{k=n^T}^{n^C-1}f_k^C.
\]
Note that if there are no trucks, $\rho^T=0$, the equilibrium distribution for the cars is $\f^{\eq,C}=[0,\dots,0,\rho^C]$.
Since now the equilibrium distributions are known, we obtain the total flux of vehicles in the case $R\leq 1/2$,
\begin{equation}
q(\rho^C,\rho^T) =	\rho^T v_{n^T} + \sum_{j=n^T}^{n^C} v_j f^C_j.
	\label{eq:free_flow}
\end{equation}
Therefore the flux depends not only on $R$, but also on the composition of the mixture. When the critical value $R=1/2$ is crossed, $f^C_j,\,f^T_j$, $j=1,\dots,n^T-1$ are turned on, meaning that there are vehicles at lower velocities. This  leads to a decrease in the flow values. 

We conclude that the maximum flow is found for $R=1/2$, which means that the critical space, across which the phase transition occurs, is given by $s_c=\left(1/2\right)^{1/\gamma}$.  Thus the critical space depends on the particular $\gamma$-law chosen. 

The maximum traffic flow is obtained at $R=1/2$, for $\rho^T=0$, and it corresponds to $\vmax\rho_{\max}^C/2^{1/\gamma}$, because then $R=1/2$ implies that the flow is composed only of cars travelling at maximum speed, thus the slope of the fundamental diagram is strongly dependent on $\vmax$. 
A more detailed study of equilibria is shown in~\cite{puppoUNPUBL}.

%

\section{Fundamental diagrams of the two-population model}
\label{sec:fundamental}
In this section we investigate numerically the fundamental diagrams resulting from the two-population kinetic model~\eqref{eq:model.2pop}. As we will see, they do not only capture the main qualitative features of the experimental diagrams of Fig.~\ref{fig:exp_diag}, including especially the data dispersion in the congested flow regime, but they also provide tools to better understand the behavior of traffic at the macroscopic scale. 

In all cases studied, system~\eqref{eq:model.2pop} is integrated numerically up to equilibrium, using the well balanced formulation~\eqref{eq:well-balanced}. Once the equilibrium distributions have been computed, the flux and the mean speed are obtained as moments of the kinetic distributions as indicated in~\eqref{eq:var.macro.p}. Since in the space homogeneous case the total density $\rho=\sum_\p\rho^\p$ is constant in time, it acts as a parameter, fixed by the initial condition, characterizing the  macroscopic quantities.

As a matter of fact, each $\rho^\p$ is also constant in time, therefore the fraction of road occupancy $s$ defined in~\eqref{eq:s} remains also stationary. It is then possible to study the flux and mean speed at equilibrium as functions of the density, and also as functions of $s$. Summarizing, we will study two types of equilibrium diagrams:
\begin{itemize}
\item \emph{Flux-density diagrams}, that is diagrams relating the total flux at equilibrium $q=\sum_\p q^{\eq,\p}=\sum_\p\sum_{j=1}^{n^\p}v_jf^{\eq,\p}_j$ to the total density $\rho=\sum_\p\rho^\p$, which corresponds to the total number of vehicles per unit length, irrespective of the size of the different vehicles. Experimental diagrams are indeed expected to represent such a relationship.
\item \emph{Flux-space diagrams}, that is diagrams relating the total flux at equilibrium to the fraction of road occupancy $s$.
\end{itemize}

Except when otherwise stated, all simulations are performed with the parameters indicated in Table~\ref{tab:parameters}, expressed in kilometers. Initially we consider $n^C=3$ speed classes for cars and $n^T=2$ speed classes for trucks, hence the corresponding spaces of microscopic speeds are
\[
	\mathcal{V}^C=\{0,\,50~\unit{km/h},\,100~\unit{km/h}\}, \qquad
	\mathcal{V}^T=\{0,\,50~\unit{km/h}\},
\]
cf. Section~\ref{sec:2pop.kin}.

\begin{table}[!t]
\caption{Parameters of model~\eqref{eq:model.2pop} common to all simulations.}
\label{tab:parameters}
\begin{center}
\begin{tabular}{clc}
Parameter & Description & Value \\
\hline
\hline
$\alpha$ & Environmental parameter & $1$ \\
$l^C$ & Typical length of a car & $4~\unit{m}$ \\
$l^T$ & Typical length of a truck & $12~\unit{m}$ \\
$\rhomax^C$ & Maximum car density & $250~\unit{vehicles/km}$ \\
$\rhomax^T$ & Maximum truck density & $83.3~\unit{vehicles/km}$ \\
$\vmax$ & Maximum speed & $100~\unit{km/h}$ \\
\hline
\end{tabular}
\end{center}
\end{table}

\begin{table}[!t]
\caption{Deterministic pairs $(\rho^C,\,\rho^T)$ used in the fundamental diagrams of Figs.~\ref{fig:ing_nornd}--\ref{fig:num_nornd} for given values of the fraction of road occupancy $s$.}
\label{tab:combinations}
\begin{center}
\begin{tabular}{lllcc}
Combination type & Marker & Expression & $\rho^C$ & $\rho^T$ \\
\hline
\hline
Space occupied mostly by cars & Crosses & $\rho^Tl^T=\frac{1}{2}\rho^Cl^C$ & $\frac{2s}{3l^C}$ & $\frac{s}{3l^T}$ \\
Space evenly occupied by cars and trucks & Circles & $\rho^Tl^T=\rho^Cl^C$ & $\frac{s}{2l^C}$ & $\frac{s}{2l^T}$ \\
Space occupied mostly by trucks & Dots & $\rho^Tl^T=2\rho^Cl^C$ & $\frac{s}{3l^C}$ & $\frac{2s}{3l^T}$ \\
\hline
\end{tabular}
\end{center}
\end{table}

\begin{figure}[!t]
\centering
\includegraphics[width=\textwidth]{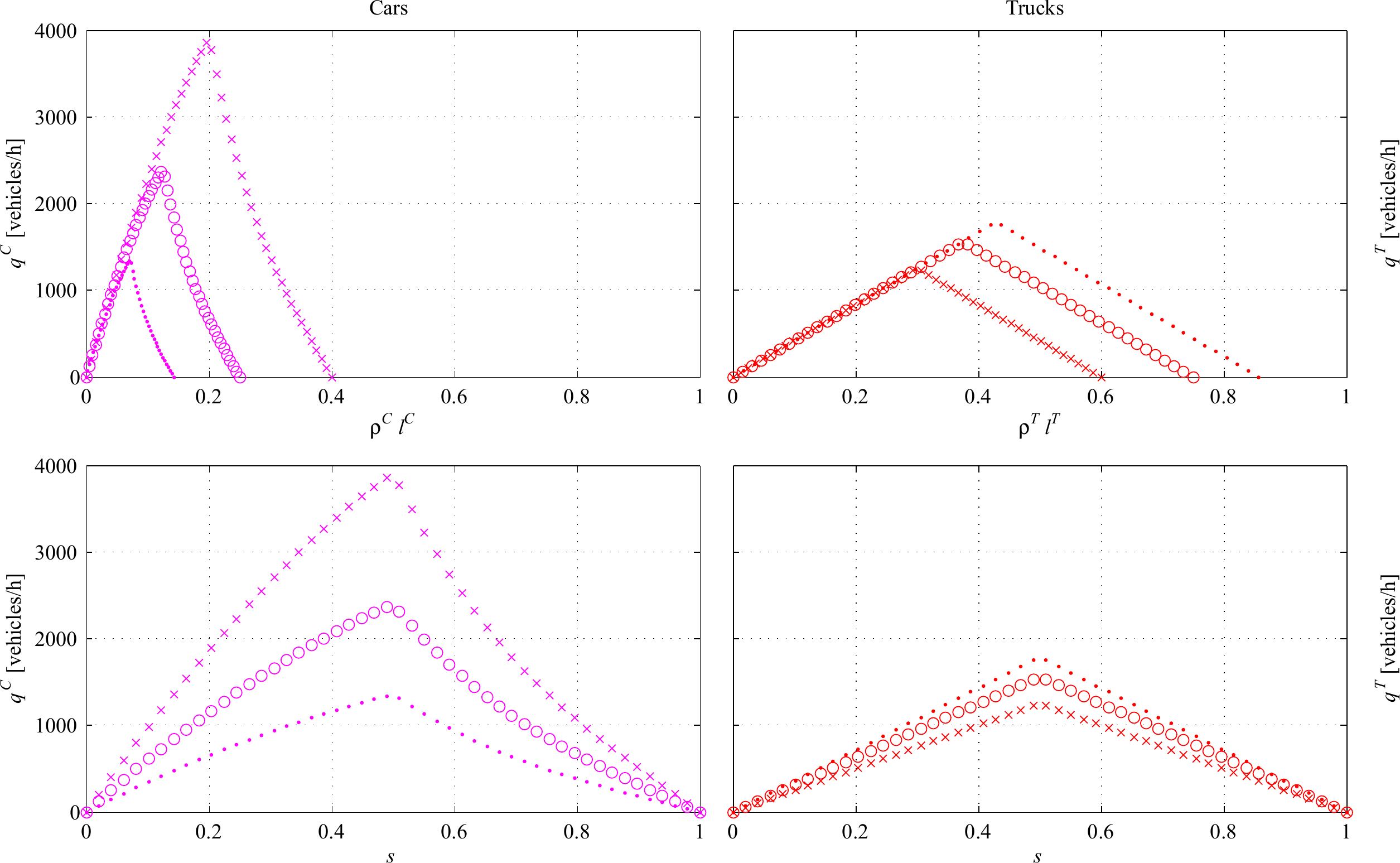}
\caption{Flux-space diagrams for the three conditions of road occupancy listed in Table~\ref{tab:combinations}.}
\label{fig:ing_nornd}
\end{figure}

In Fig.~\ref{fig:ing_nornd} we show the flux-space diagrams of each class of vehicles obtained using deterministic initial conditions: for each $s\in [0,\,1]$ we select three prototypical pairs $(\rho^C,\,\rho^T)\in [0,\,\rhomax^C]\times [0,\,\rhomax^T]$ such that $\rho^Cl^C+\rho^Tl^T=s$, corresponding to different conditions of road occupancy, cf. Table~\ref{tab:combinations}. The resulting fundamental diagrams are qualitatively similar to those obtained from the single-population model~\eqref{eq:kinetic.1pop} with analogous numbers of speed classes. For instance, the fundamental diagram of the trucks alone compares well with the one shown in Fig.~\ref{fig:tosin} with $n=2$ speed classes.

All plots in Fig. \ref{fig:ing_nornd} show clearly that there is a critical fraction of occupied space, beyond which the flow starts to decrease.
On the top section of the figure, the two subplots show that the critical space for each species changes depending on the mixture we consider. In fact, the space occupied by a class of vehicles is only one contribution to the fraction of occupied space which determines the transition matrices. In other words, even the dynamics of a single species depends on the dynamics of the complete mixture. Consequently, the flow depends on the composition of traffic. 

In the bottom section of the figure,  the flow of cars and trucks is shown as a function of the total fraction of occupied space $s$. One can immediately note that there is a single value for the critical space which corresponds to $s_c=\tfrac12$ for all three combinations. This result seems to suggest that the transition from the free to the congested phase  does not depend on how the road is occupied but on how much of it is occupied. This value of $s_c$ depends on the particular choice $\gamma=1$ in the expression of the probability $P$, see \S \ref{sec:equilibria}.


\begin{figure}[!t]
\centering
\includegraphics[width=\textwidth]{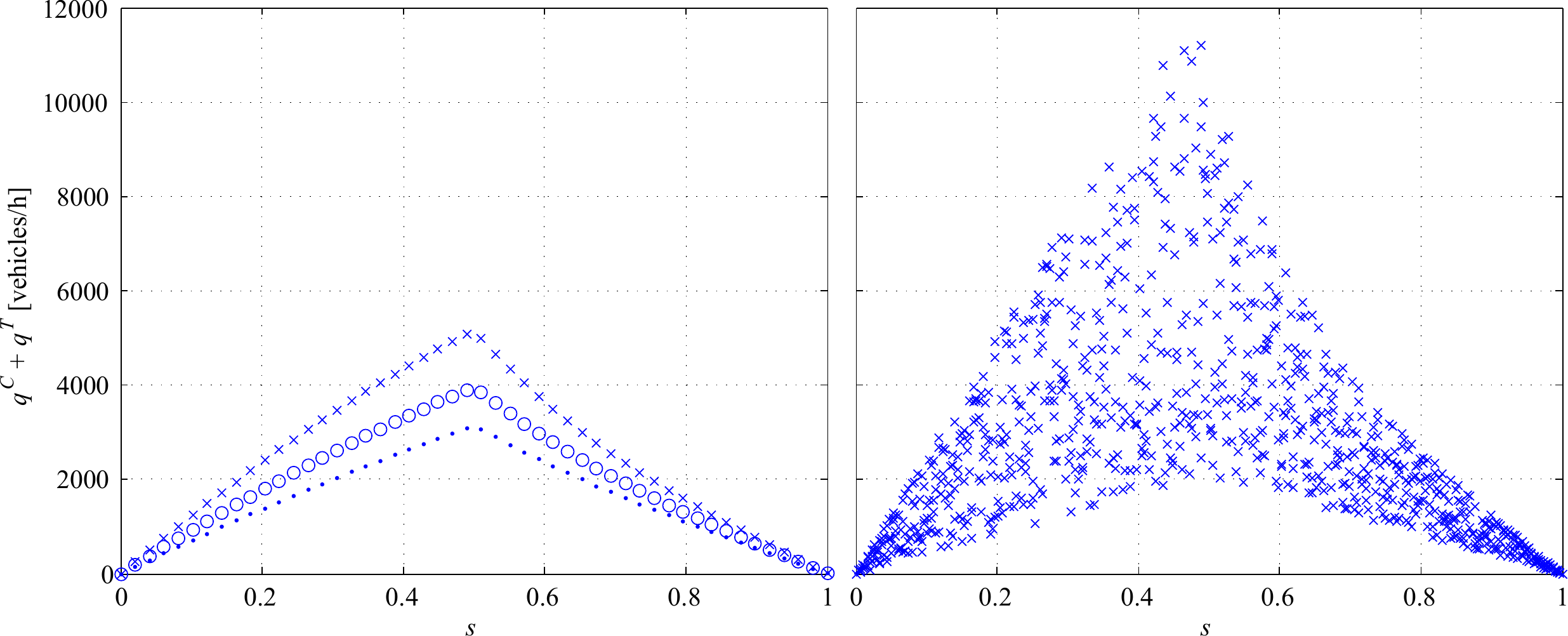}
\caption{Flux-space diagrams. Left: deterministic choice of the pairs $(\rho^C,\,\rho^T)$ according to Table~\ref{tab:combinations}. Right: random choice of the pairs $(\rho^C,\,\rho^T)$ for each $s\in [0,\,1]$.}
\label{fig:DiagFond_ing}
\end{figure}

In Fig.~\ref{fig:DiagFond_ing} we compare the fundamental diagrams obtained by using either the three deterministic pairs $(\rho^C,\,\rho^T)$ given in Table~\ref{tab:combinations} or three pairs chosen randomly for each $s$. 
The left of Fig. \ref{fig:DiagFond_ing} shows the total flux as a function of $s$, again, for the three combinations of Table \ref{tab:combinations}, and for three random combinations, for each fixed $s$. Here the role of the critical value $s=\tfrac12$ is even more apparent.  In spite of the apparent data dispersion, this diagram does not reproduce the experimental data, because the information brought by the fraction of road occupancy $s$ is too synthetic to take into account the heterogeneity of traffic. 

Motivated by this argument, 
now we turn to flux-density diagrams, which give the flux as a function of the number of vehicles per kilometer. Indeed experimental fundamental diagrams are expected to result out of this type of observations.
In this case, the composition of traffic is taken into account, because the same fraction of occupied space $s\in[0,1]$ can be obtained by different initial densities $\rho_C,\,\rho_T$. 


\begin{figure}[!t]
\centering
\includegraphics[width=\textwidth]{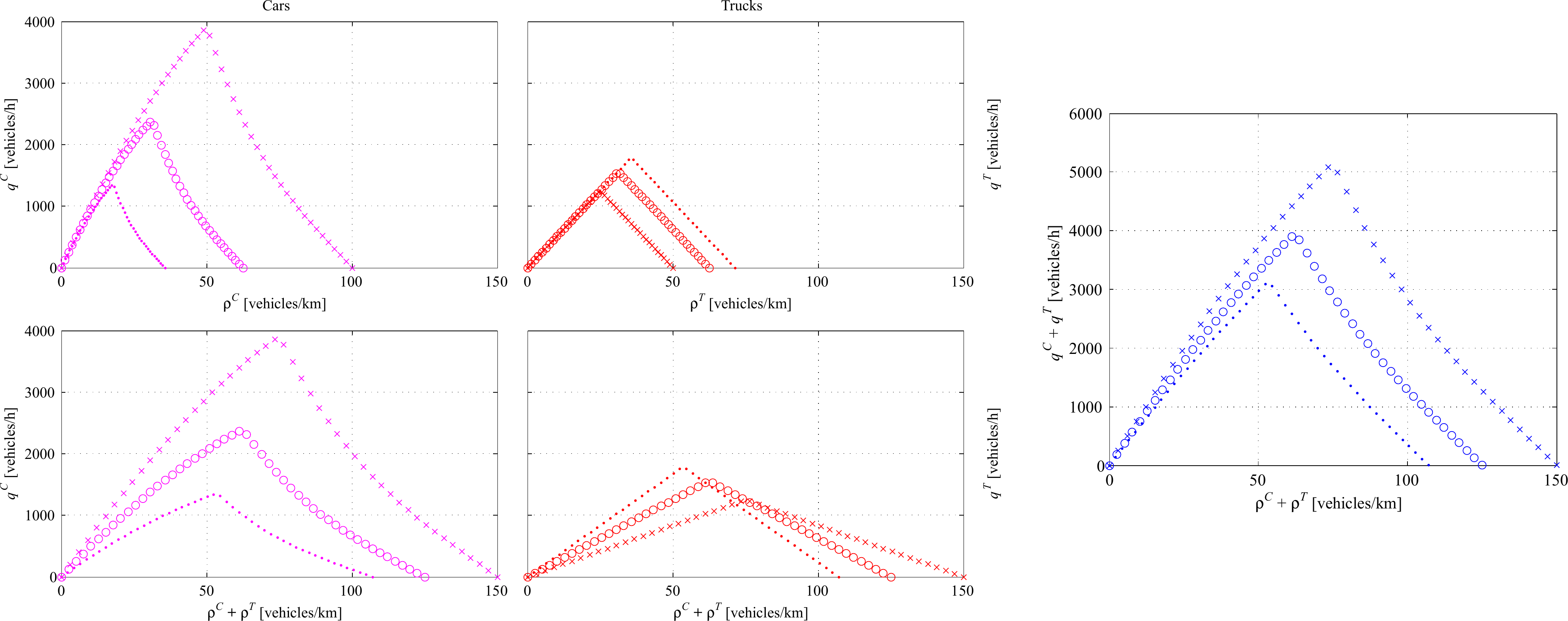}
\caption{Flux-density diagrams for the three conditions of road occupancy listed in Table~\ref{tab:combinations}.}
\label{fig:num_nornd}
\end{figure}

Again, for each $s\in [0,\,1]$ the graphs in Fig.~\ref{fig:num_nornd} are obtained by taking three pairs $(\rho^C,\,\rho^T)$ corresponding to the combinations reported in Table~\ref{tab:combinations}. 
The plots on the left correspond to the flux of each single species, as a function of its corresponding number density (top) and of the total number of vehicles (bottom).
The plot on the right gives the total flux as a function of the total number of vehicles. This deterministic choice allows us to look at the transition from free to congested phase. Here, each combination has a different critical value of the density for the phase transition, which depends on the ratio of the different species within the mixture. But we know from Fig. \ref{fig:DiagFond_ing} that each of these critical values of the density will correspond to the single value  $s=\tfrac12$. Note that the plot on the right begins to resemble the experimental fundamental diagrams of Fig. \ref{fig:exp_diag}. 

\begin{figure}[!t]
\centering
\includegraphics[width=\textwidth]{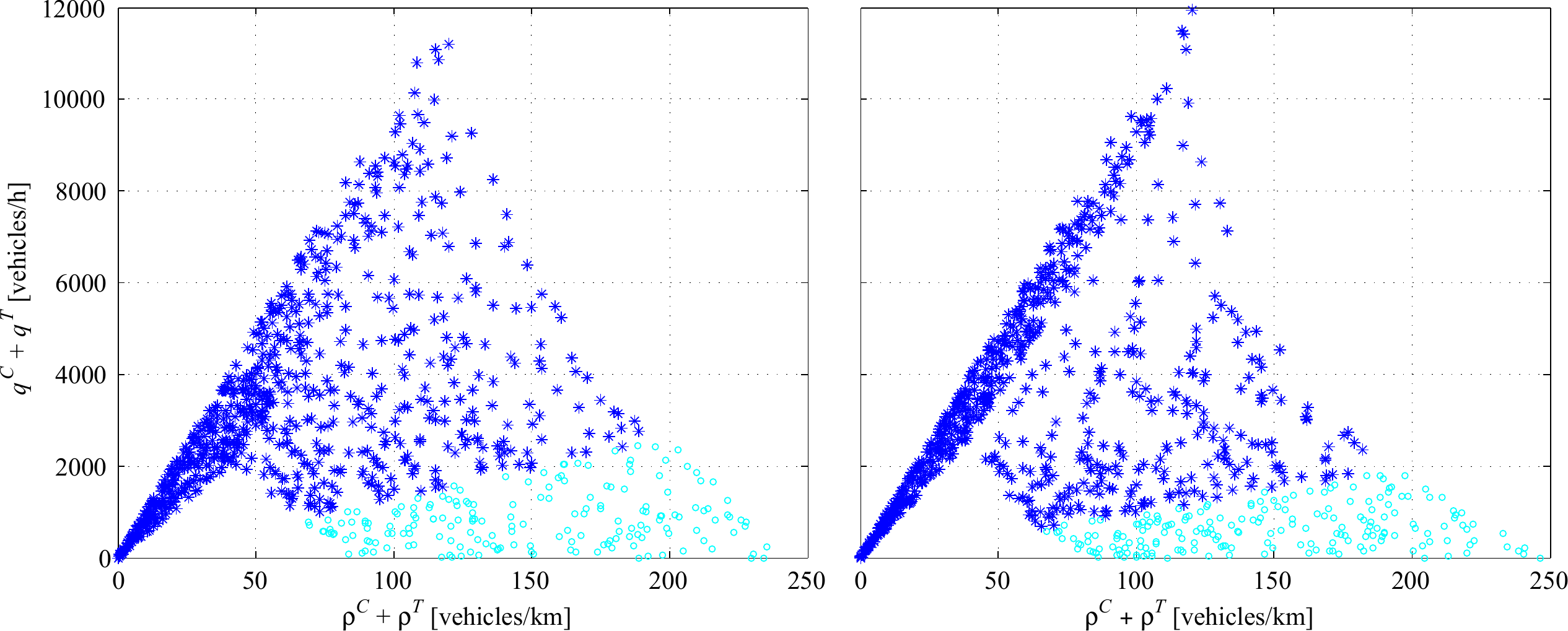}
\caption{Flux-density diagram of the complete mixture obtained with three random pairs $(\rho^C$,\,$\rho^T)$ for each $s\in [0,\,1]$. Left: $n^C=3$, $n^T=2$; right: $n^C=4$, $n^T=3$. Blue star markers: data for $s\leq 0.8$; cyan circle markers: data for $s>0.8$.}
\label{fig:25Feb}
\end{figure}

By sampling three random pairs $(\rho^C,\,\rho^T)$ for any given $s\in [0,\,1]$ we obtain the fundamental diagrams illustrated in Fig.~\ref{fig:25Feb}, which clearly capture the main characteristics of the experimental diagrams discussed in Section~\ref{sec:fund_diag}. In particular, at low densities the total flux grows nearly linearly with small dispersion, while at higher densities it decreases with larger dispersion due to the frequent interactions between fast and slow vehicles. In the graph, cyan circles indicate the total density-total flux pairs obtained for $s\in (0.8,\,1]$, whereas blue stars indicate those obtained for $s\in [0,\,0.8]$. As a matter of fact, the latter are the most likely to occur in practice, since even in traffic jams vehicles attain seldom a state of maximum density and complete stop (see e.g., Fig.~\ref{fig:exp_diag}, where a residual movement always appears). Note also that the plot on the right exhibits a capacity drop across the critical density. The behaviour of the model with respect to the number of discrete velocities is analyzed in~\cite{puppoUNPUBL}.

\begin{figure}[!t]
\centering
\includegraphics[width=\textwidth]{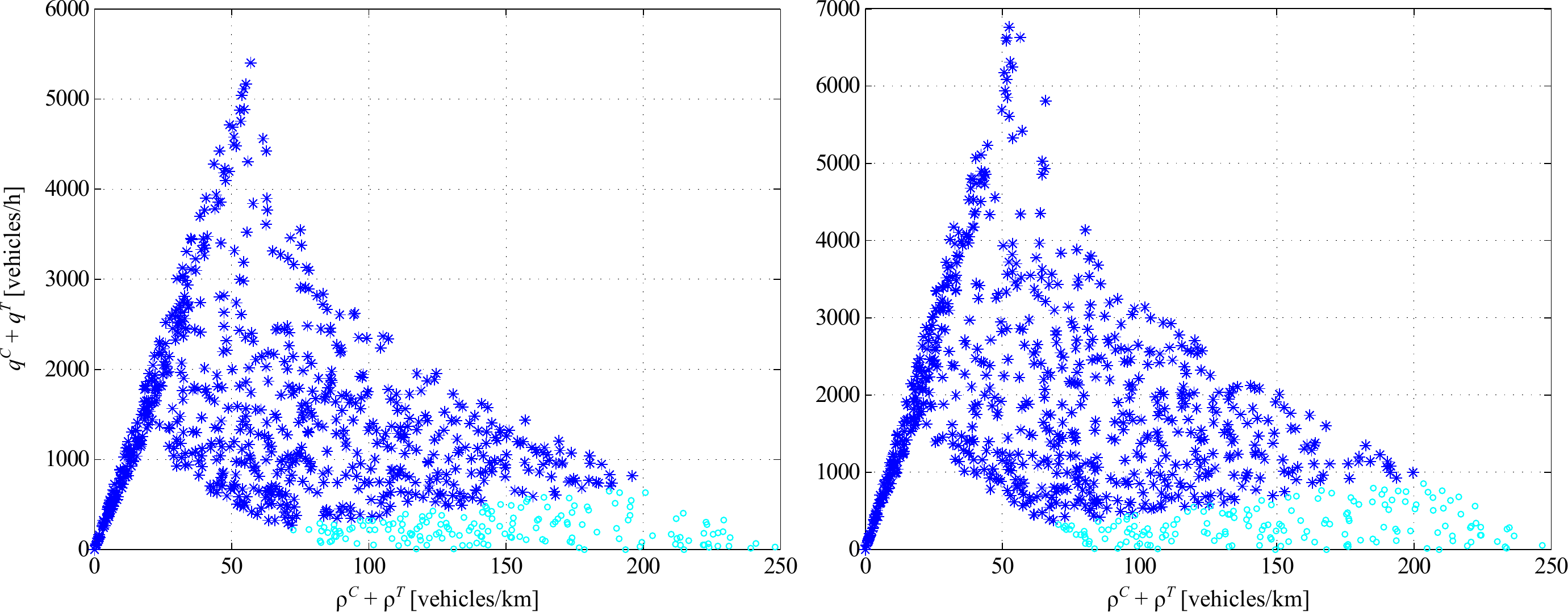}
\caption{Fundamental diagrams with $n^C=4$ and $n^T=3$ velocity classes and probability transition $P=1-\sqrt{s}$. The maximum speeds $v_{n^C}=100~\unit{km/h}$ (on the left) and $v_{n^C}=130~\unit{km/h}$ (on the right) are considered. The diagrams are obtained with three random pairs $(\rho^C$,\,$\rho^T)$ for each $s\in [0,\,1]$. Blue star markers: data for $s\leq 0.8$; cyan circle markers: data for $s>0.8$.}
\label{fig:square}
\end{figure}

In Figures~\ref{fig:ing_nornd}-\ref{fig:25Feb}  the transition probabilities were $P=1-s$ and $Q=0$, as given in~\eqref{eq:probab_s}. This choice determines flux-space diagrams in which the transition from free to congested phase occurs at the value $s=1/2$ for any composition of the mixture, see Figures~\ref{fig:ing_nornd}, \ref{fig:DiagFond_ing} and \S \ref{sec:equilibria}. This leads to fundamental diagrams in which the maximum value of the flow is reached when $\rho^T=0$ and $\rho=\rho^C_{\max}/2=125~\unit{vehicles/km}$, see Figures~\ref{fig:num_nornd} and~\ref{fig:25Feb}, because in this case the flow is composed only of cars travelling at their maximum speed.
 Choosing instead $P=1-s^{\gamma}$, with $\gamma<1$,  the phase transition occurs at $s_c=\left(\frac12\right)^{\frac{1}{\gamma}}$ and it decreases when $\gamma$ is decreased. In particular, in the left plot of Figure~\ref{fig:square} we consider $\gamma=1/2$ and we observe that the fundamental diagram, obtained with $n^C=4$ and $n^T=3$, shows a better reproduction of experimental data, see Section~\ref{sec:fund_diag}, Figure~\ref{fig:exp_diag}. In the right plot, instead, we consider $\gamma=1/2$ but $\vmax=130~\unit{km/h}$, showing that a greater maximum speed causes an increase in the slope of the diagram in the free flow phase.

\begin{figure}[!t]
\centering
\includegraphics[width=\textwidth]{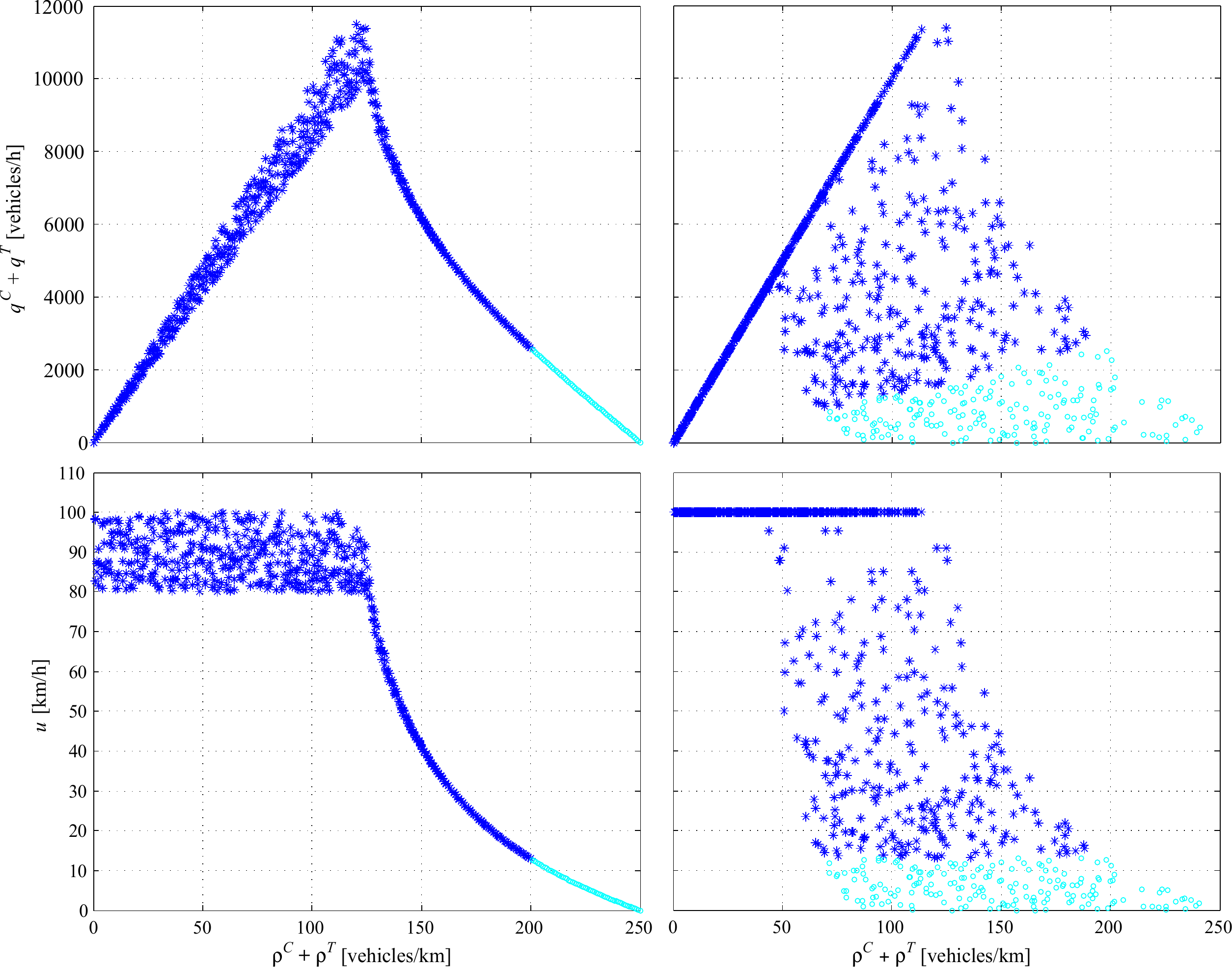}
\caption{Top row: flux-density diagrams, bottom row: speed-density diagrams vs. the total density for two populations of vehicles having either the same length and different microscopic speeds (left) or different lengths and same microscopic speeds (right). Blue star markers: data for $s\leq 0.8$; cyan circle markers: data for $s>0.8$.}
\label{fig:char_simili}
\end{figure}

The examples discussed so far suggest that the bulk characteristics of traffic at equilibrium could be predicted \emph{deterministically} once the composition of traffic, i.e., the pair $(\rho^C,\,\rho^T)$, is known. This induces to interpret the scattering of data in the congested phase as a consequence of 
the possible heterogeneity of vehicles in traffic, for a given level of road occupancy, rather than as an effect of the unpredictability of driver behaviours.

This affirmation can be articulated more precisely, by considering the 2-populations model  in which vehicles differ by only one characteristic. The plot on the left of Figure \ref{fig:char_simili} shows the flux-density diagram when the two classes of vehicles have the same length, but they differ  in their maximum speed: $\mathcal{V}^C=\left\{0,50,80,100\right\}$ and $\mathcal{V}^T=\left\{0,50,80\right\}$. We can interpret this case as thinking that vehicles are now identical, but we are considering two different types of driver, according to the maximum speed they are willing to settle on when the road is free (say, fast and slow drivers).
The plot on the right of Figure \ref{fig:char_simili} is obtained by considering vehicle classes which have different lengths, as given in Table~\ref{tab:parameters}, but the same microscopic  speeds. 

By inspecting them we infer that differences in the speeds (in particular, the maximum ones) of the vehicles composing the traffic mixture seem to be responsible for the small scattering of the data in the free flow phase, whereas differences in the length determine the larger scattering of the data in the congested flow phase.

\begin{figure}[!t]
\centering
\includegraphics[width=\textwidth]{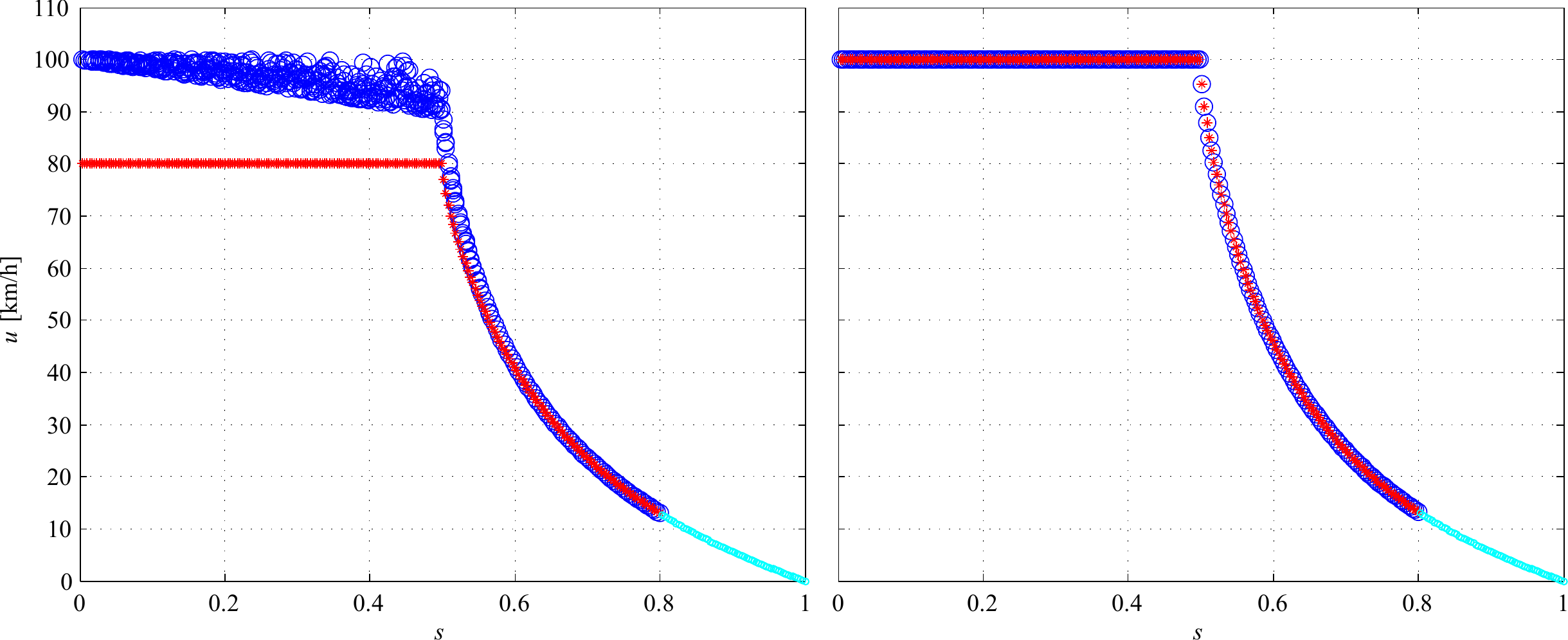}
\caption{Mean speed versus fraction of road occupancy $s$ for cars (blue) and trucks (red) having either the same length and different microscopic speeds (left) or different lengths and same microscopic speeds (right). Cyan circle markers refer to data for $s>0.8$.}
\label{fig:vel_char_simili}
\end{figure}

The same two cases are further investigated in Fig.~\ref{fig:vel_char_simili} by focusing on the speed diagram vs. the fraction of road occupancy $s$. In particular, when vehicles have different microscopic speeds but same length (diagram on the left) we deduce that, in free flow conditions, the slower population is not affected by the faster one, while fast drivers may have to slow down due to their interactions with slower cars. On the other hand, both types of vehicles are forced to slow down, reaching finally the same mean speed, as the road becomes congested. Conversely, when vehicles have the same microscopic speeds but different lengths (diagram on the right) we discover that the mean speed is the same for both populations in both traffic regimes, i.e., in other words, it is a one-to-one function of the fraction of occupied space.

These remarks are indeed consistent with daily experience of driving on highways: in free flow drivers can choose their speed, and thus they keep different maximum speeds according to their driving style, while in congested flow they tend to travel all at the same speed, which steadily decreases as the traffic congestion increases.

\begin{figure}[!t]
\centering
\includegraphics[scale=0.5]{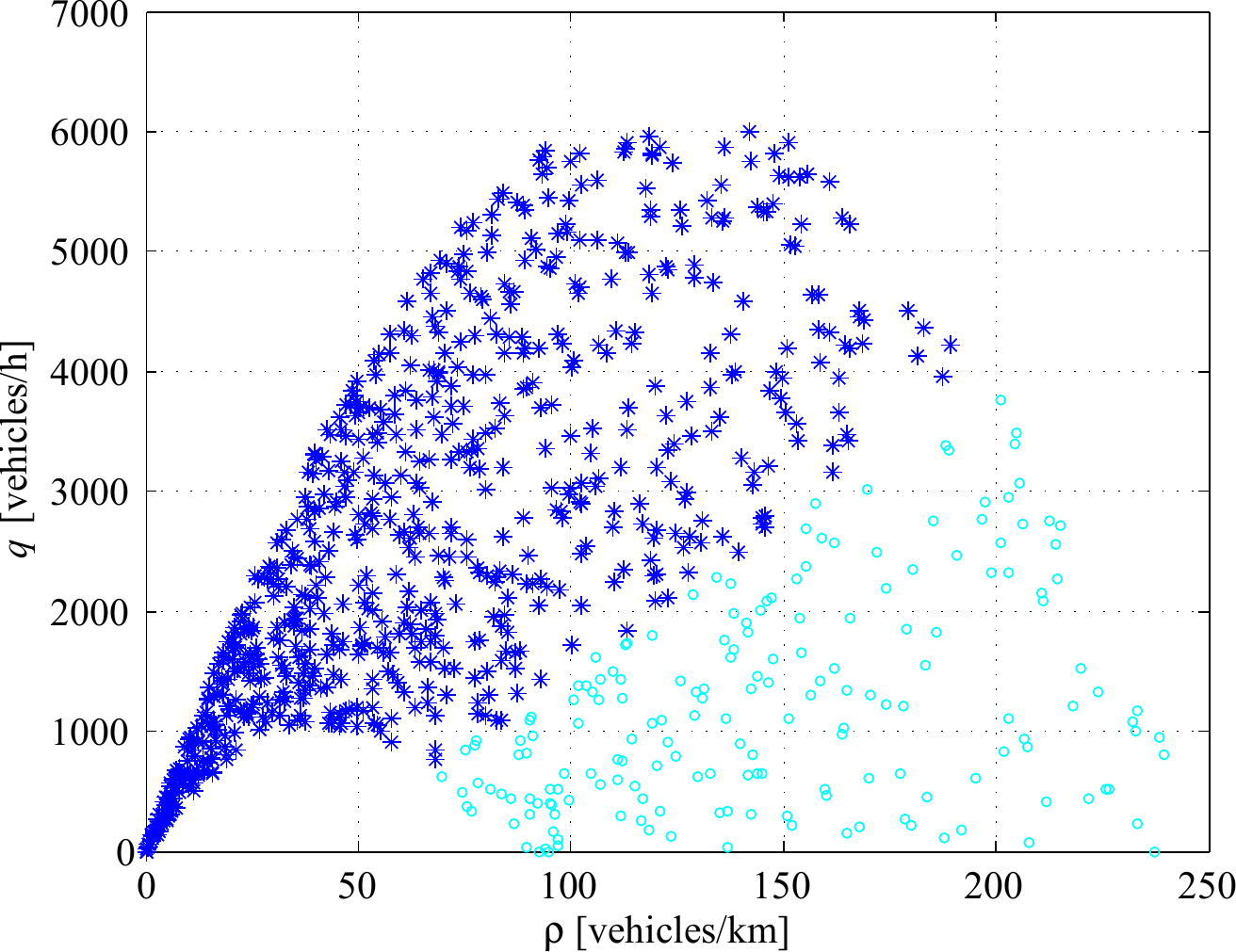}
\caption{Flux-density diagram of the two-population macroscopic model~\cite{benzoni2003EJAM}, see also~\eqref{mod:colombo2}, obtained by sampling three random pairs $(\rho_1,\,\rho_2)$ for every $s\in [0,\,1]$. Blue star markers: data for $s\leq 0.8$; cyan circle markers: data for $s>0.8$.}
\label{fig:colombo}
\end{figure}

Finally, 
Figure~\ref{fig:colombo} shows the fundamental diagram for the traffic mixture modelled by the two-population macroscopic model~\cite{benzoni2003EJAM} summarized in Section~\ref{sec:2pop}. It is immediate to notice that there is no trace of the sharp phase transition predicted by our kinetic model and that the scattering of the data is very high also at low densities.

\section{Conclusions and perspectives}
\label{sec:conclusions}
In this paper we have introduced a kinetic model for vehicular traffic with a new structure which accounts for the heterogeneous composition of traffic flow. Our approach differs from standard kinetic models in that  we consider two distribution functions describing two classes of vehicles with different physical features, in this case the typical length of a vehicle and its maximum speed.

As in~\cite{fermo2013SIAP}, the model is built by assuming a discrete space of microscopic speeds and by expressing vehicle interactions in terms of transition probabilities among the admissible speed classes. We have shown that our two-population model satisfies an indifferentiability principle, which makes it consistent with the original single-population model when the particles composing the mixture share the same physical characteristics (in our case the vehicle length and maximum speed). This property, enforced in~\cite{andries2001REPORT}, is not trivial, and several kinetic models for gas mixtures possess it only at equilibrium~\cite{groppi2004PF,hamel1965PF}.

We have then used our two-population kinetic model to perform a computational analysis of the equilibria of the system and to derive in this way the fundamental diagrams predicted by the simulated dynamics. Even with a small number of microscopic speeds, such diagrams feature a structure closely resembling experimental data. In particular, they are characterised by a marked phase transition: at low vehicle densities (free flow) the flux increases almost linearly with small standard deviation, while beyond a critical density the flow decreases taking widely scattered values (congested flow). We have also computed the critical density at which the transition occurs.

Several authors have dealt with this problem, cf. e.g.,~\cite{fermo2014DCDSS,gunther2004SIAP} where this phenomenon is explained by invoking the uncertainty of the drivers' behaviour in terms of standard deviation of the statistical distribution of speeds at equilibrium. However, such an approach predicts a zero standard deviation in the free phase of traffic and furthermore interprets the scattered distribution of the data in the congested phase as a consequence of the variability of the \emph{microscopic} speeds at equilibrium. In our case, instead, we do not only recover the sharp phase transition, which seems to result naturally from our kinetic approach, but we also obtain the scattered behaviour at a genuinely \emph{macroscopic} level as a consequence of the fact that a given road occupancy can be obtained with different compositions of the mixture.  In other words, if the flux is given as a function of the number of vehicles crossing a section of road in a unit time, then our model indicates that the scattering may be due to the simultaneous presence of different types of vehicles. On the other hand, in the congested phase the mean speed of the vehicles seems to depend only on the degree of congestion of the road.


Finally, we also wish to note that the model is very simple: the complexity of a real flow is clustered in the characteristics of only two distinct populations, with a very small number of microscopic velocities. Thus, from a computational point of view, this construction is not significantly more demanding than a macroscopic model.

As far as the analytical properties of the model are concerned, we refer to the forthcoming paper~\cite{puppoUNPUBL}, where we prove the well-posedness of the Cauchy problem associated with~\eqref{eq:model.2pop}, in the sense that the solution exists, is unique, depends continuously on the initial data, and moreover remains nonnegative and bounded by the initial mass. Furthermore, we can also prove that equilibria, which define the fundamental diagrams, are uniquely determined by the initial mass of the two classes of vehicles, and, in some simplified cases, they can be computed explicitly. Additional study will be dedicated to the extension of the present model to road networks and multilane highways.

\bigskip
\begin{acknow}
This work was partly supported by \textquotedblleft National Group for Scientific Computation\textquotedblright (GNCS-INDAM)
\end{acknow}

\bibliographystyle{plain}
\bibliography{PgSmTaVg-funddiag_traffic}

\end{document}